\documentclass[reqno,11pt]{amsart}
\usepackage{amsmath,amssymb,latexsym,soul,cite,mathrsfs,accents}
\usepackage{color,enumitem,graphicx}
\usepackage{xcolor}
\usepackage[colorlinks=true,urlcolor=blue,
citecolor=red,linkcolor=blue,linktocpage,pdfpagelabels,
bookmarksnumbered,bookmarksopen]{hyperref}

\usepackage{mathtools}

\usepackage[hyperpageref]{backref}
\usepackage[english]{babel}
\usepackage[left=2.6cm,right=2.6cm,top=2.9cm,bottom=2.9cm]{geometry}
\usepackage{tensor}
\usepackage{tikz,pgfplots}
\pgfplotsset{compat=1.18}
\usepackage{cancel}
\usepackage{esint}
\newtheorem{theorem}{Theorem}
\newtheorem{lemma}[theorem]{Lemma}
\newtheorem{corollary}[theorem]{Corollary}
\newtheorem{proposition}[theorem]{Proposition}
\newtheorem{remark}[theorem]{Remark}

\newtheorem{theoremletter}{Theorem}

\newcommand{\innerthmname}{}

\theoremstyle{definition}

\makeatletter
\def\namedlabel#1#2{\begingroup
	#2%
	\def\@currentlabel{#2}%
	\phantomsection\label{#1}\endgroup
}

\newcommand*\owedge{\mathpalette\@owedge\relax}
\newcommand*\@owedge[1]{%
	\mathbin{%
		\ooalign{%
			$#1\m@th\bigcirc$\cr
			\hidewidth$#1\m@th\wedge$\hidewidth\cr
		}%
	}%
}
\makeatother

\newcommand{\ud}{\,\mathrm{d}}

\newcommand{\R}{\mathbf{R}}

\newcommand{\N}{\mathbf{N}}
\newcommand{\Ss}{\mathbf{S}}

\definecolor{TODOcolor}{RGB}{200,50,50}   
\definecolor{DONEcolor}{RGB}{40,150,40}    

\newif\ifdraft
\draftfalse  

\ifdraft
  \usepackage[colorinlistoftodos,prependcaption,textsize=footnotesize]{todonotes}
\else
  \usepackage[disable]{todonotes} 
\fi

\ifdraft
  \usepackage{pdfcomment} 
\else
  \usepackage[final]{pdfcomment} 
\fi

\title[Quantitative stability for the Trudinger--Moser inequality]{Quantitative stability for the Trudinger--Moser inequality} 
\thanks{J.H.A.\ was partially supported by the S\~ao Paulo Research Foundation (FAPESP), grant \#2023/15567-8, and by the National Council for Scientific and Technological Development (CNPq), grants \#409764/2023-0, \#443594/2023-6, \#441922/2023-6, and \#306014/2025-4.
J.F.O.\ was partially supported by CNPq, grants \#309491/2021-5 and \#303443/2025-1.
J.M.O.\ was partially supported by CNPq, grants \#312340/2021-4, \#409764/2023-0, and \#443594/2023-6, by CAPES MATH AMSUD, grant \#88887.878894/2023-00, and by the Para\'iba State Research Foundation (FAPESQ), grant \#3034/2021.
A.C.M.\ was partially supported by CNPq, grants \#443594/2023-6 and \#441922/2023-6.
J.R.\ was partially supported by the Deutsche Forschungsgemeinschaft (DFG), grant \#561401741.}

\author[J.H. Andrade]{Jo\~{a}o Henrique Andrade}
\author[J.F. de Oliveira]{Jos\'e Francisco de Oliveira}
\author[J.M. do \'O]{Jo\~{a}o Marcos do \'O}
\author[A. C. Macedo]{Abiel Costa Macedo}
\author[J. Ratzkin]{Jesse Ratzkin}

\address[J.H. Andrade]{
	Institute of Mathematics and Statistics,
	University of S\~ao Paulo
	\newline\indent
	05508-090, S\~ao Paulo-SP, Brazil
	}
\email{\href{mailto:andradejh@ime.usp.br}{andradejh@ime.usp.br}}

\address[J.F. de Oliveira]{
	Department of Mathematics,
	Federal University of Piau\'{\i}
	\newline\indent
	64049-550 - Teresina, PI - Brazil
	}
\email{\href{mailto:jfoliveira@ufpi.edu.br}{jfoliveira@ufpi.edu.br}}

\address[J.M. do \'O]{
	Department of Mathematics,
	Federal University of Para\'{\i}ba
	\newline\indent
	58051-900 - Jo\~ao Pessoa, PB - Brazil
	}
\email{\href{mailto:jmbo@academico.ufpb.br}{jmbo@academico.ufpb.br}}

\address[A. C. Macedo]{
	Institute of Mathematics and Statistics,
	Federal University of Goi\'as
	\newline\indent
	74690-900 - Goi\~ania, GO - Brazil
	}
\email{\href{mailto:abiel@ufg.br}{abiel@ufg.br}}

\address[J. Ratzkin]{Department of Mathematics,
	Universit\"{a}t W\"{u}rzburg
	\newline\indent
	97070, W\"{u}rzburg-BA, Germany}
\email{\href{mailto:jesse.ratzkin@mathematik.uni-wuerzburg.de}{jesse.ratzkin@mathematik.uni-wuerzburg.de}}

\subjclass[2020]{35B35, 35J20, 35J60, 46E35, 35B45}
\keywords{Trudinger--Moser inequality; quantitative stability; Bianchi--Egnell inequality; nonlinear eigenvalue problems}

\begin{document} 

\begin{abstract}
We establish quantitative stability estimates for the Trudinger--Moser inequality on smooth, 
bounded domains of the Euclidean plane. More 
specifically, we prove that the deficit in the Trudinger--Moser functional quadratically controls the distance 
to the set of optimizers if either (i) the exponential rate of growth is sufficiently small 
or (ii) the domain is a 
round disk. The latter estimate remains valid even in the critical case. Both proofs rely on 
a new spectral gap that we prove, which may be of independent 
interest. Additionally, we show the same stability estimate in the nondegenerate case, and show this 
occurs generically.  
\end{abstract}

\maketitle

\section{Introduction}

Let $n\geq 2$ be a natural number. The classical Sobolev embedding theorem 
states that for each $q \in (1,n)$ the Sobolev 
space $W^{1,q}(\R^n)$ embeds into $L^{\frac{nq}{n-q}}(\R^n)$. Trudinger \cite{Trudinger67}
showed that this embedding fails in the critical case $q=n$, but that functions 
in $W^{1,n}_0(\Omega)$ are of exponential type for any bounded domain $\Omega\subseteq\mathbf{R}^n$ 
with smooth boundary $\partial\Omega\in \mathcal{C}^\infty$. Moser \cite{Moser1970/71} later 
sharpened this result, determining the 
best possible rate of exponential growth. The combination of these two 
results is now widely known as the Trudinger--Moser inequality, which states 
\begin{equation} \label{moser_tru_ineq} 
{\rm TM} (n, \alpha, \Omega) = \sup_{\substack{u \in W^{1,n}_0(\Omega)\\ 
\| \nabla u \|_{L^n (\Omega)}\leq 1}} \left \{ \mathscr{E}_{n, \alpha, \Omega} 
(u) = \int_\Omega e^{\alpha u^{\frac{n}{n-1}}} \ud x \right \} < \infty \Leftrightarrow 
0 \leq \alpha \leq \alpha_*(n) = n \omega_{n-1}^{n-1}.
\end{equation}
Here $\omega_{n-1} = |\Ss^{n-1}|$ is the surface measure of the $(n-1)$-dimensional sphere. Observe 
that we implicitly define the Trudinger--Moser functional $\mathscr{E}_{n,\alpha,\Omega}$ 
in \eqref{moser_tru_ineq}.

For any given $\Omega \subset \R^n$, 
the set of extremal functions ({\it i.e.,} those functions realizing equality 
in \eqref{moser_tru_ineq}) is nonempty. Carleson and Chang \cite{Carleson-Chang}
demonstrated the existence of extremal functions in the case that $\Omega$ 
is a two-dimensional disk. Subsequently, Flucher \cite{Flucher1992} extended 
the existence result to any smooth, bounded domain in $\R^2$, and finally 
Lin \cite{Lin1996} proved the set of extremals is nonempty in 
any dimension $n \geq 2$. However, the precise form of the extremal 
functions remains unresolved, even in the case of the disk.

\subsection{Main results} In the present manuscript, we investigate the 
stability of \eqref{moser_tru_ineq}
in dimension two. More precisely, we are interested in quantifying the statement 
that if $u \in W^{1,2}_0(\Omega)$ with $\| \nabla u \|_{L^2(\Omega)} \leq 1$ 
is close to realizing the supremum in \eqref{moser_tru_ineq}, then it 
must be close to an extremal function. We measure the distance to the set of 
extremals as 
\begin{equation} \label{dist_defn} 
{\rm dist}_{\alpha,\Omega}(u):=\inf_{v\in\mathcal{M}_\alpha(\Omega)}
\|\nabla u-\nabla v\|_{L^2(\Omega)}, \quad {\rm for} \quad u\in W^{1,2}_0(\Omega),
\end{equation} 
where $\mathcal{M}_\alpha(\Omega)$ is the set of functions realizing equality 
in \eqref{moser_tru_ineq}. Notice that we measure the distance in $W^{1,2}_0(\Omega)$ 
as the $L^2$-norm of the gradient; by the Poincar\'e inequality this is equivalent 
to the standard $W^{1,2}$-norm. We also define the deficit 
\begin{equation}\label{deficit_defn} 
{\rm def}_{\alpha,\Omega}(u):={\rm TM}(\alpha,\Omega)-\mathscr{E}_{\alpha,\Omega}(u) 
\end{equation} 
and observe that in two dimensions the critical value of $\alpha$ 
is $\alpha_*(2) = 4\pi$. We denote the unit ball in $W^{1,2}_0(\Omega)$ 
by $\mathcal{B}_1$ and the unit sphere in $W^{1,2}_0(\Omega)$ by $\mathcal{S}_1$, 
once again with respect to the norm $\| \nabla \cdot \|_{L^2(\Omega)}$; that is,
\[
\mathcal{B}_1 = \{ u \in W^{1,2}_0(\Omega): \| \nabla u \|_{L^2(\Omega)}  < 1 \}
\]
and the unit sphere
\[
\mathcal{S}_1 = \{ u \in W^{1,2}_0(\Omega) : \| \nabla u \|_{L^2(\Omega)} = 1\}.
\]
Note that, since $\mathscr{E}_{\alpha,\Omega}$ is strictly increasing along radial 
rescalings $u\mapsto u/\|\nabla u\|_{L^2(\Omega)}$, every optimizer satisfies 
$\|\nabla u\|_{L^2(\Omega)}=1$, so that $\mathcal{M}_\alpha(\Omega)\subset\mathcal{S}_1$ 
(see \eqref{maxB} below).

Our first result states that when $\alpha$ is sufficiently small, the deficit 
is bounded below by the square of the distance. 
\begin{theorem}[Small subcritical case] \label{thm:subcritical_stability} 
Let $\Omega \subset \R^2$ be an open, bounded set with a smooth boundary.
There exists $\alpha_0>0$ such that for all $\alpha \in (0,\alpha_0)$, one can find 
a constant $C_*=C_*({\alpha},\Omega)>0$ such that 
\begin{equation*}
{\rm TM}({\alpha},\Omega) - \mathscr{E}_{\alpha, \Omega} (u) \geq  C_* \inf_{v \in 
\mathcal{M}_\alpha(\Omega)} {\| \nabla u - \nabla v \|_{L^2(\Omega)}^2} 
\quad {\rm for \ all} \quad u\in \overline{\mathcal{B}}_1.
\end{equation*} 
In other words, one has
\[
{\rm def}_{\alpha,\Omega}(u)\gtrsim_{\alpha,\Omega}{\rm dist}_{\alpha,\Omega}(u)^2.
\]
\end{theorem} 

Our second main result replaces the small growth hypothesis in Theorem \ref{thm:subcritical_stability} 
with nondegeneracy.  
We recall that the quantitative behavior of the functional near its maximizers 
is governed by the second variation of $\mathscr{E}_{\alpha,\Omega}$ on the unit sphere 
$\mathcal{S}_1$. We say that a maximizer $u_*$ is \emph{nondegenerate} if the second 
variation of the Trudinger-Moser functional is coercive on $T_{u_*} (\mathcal{S}_1)$. 
More precisely, $u_*$ is nondegenerate if there exists $\mu>0$ such that 
$$\mathscr{Q}_{\alpha, \Omega}(u_*)[w,w] \geq \mu \| \nabla w \|_{L^2(\Omega)} 
\quad \mathrm{for each}\quad w \in T_{u_*}(\mathcal{S}_1),$$
where $\mathscr{Q}_{\alpha,\Omega}(u_*)$ is the quadratic form given by \eqref{eq:weighted-ip}.
\begin{theorem}[Nondegenerate subcritical case]\label{thm:subcritical_stability_nondeg}
Let $\Omega \subset \R^2$ be an open, bounded set with a smooth boundary and $\alpha\in(0,4\pi)$.
If every optimizer is nondegenerate, then there exists $C^*=C^*(\alpha,\Omega)>0$ such that
\[
{\rm TM}({\alpha},\Omega) - \mathscr{E}_{\alpha, \Omega} (u) \geq  C^*\inf_{v \in
\mathcal{M}_\alpha(\Omega)} {\| \nabla u - \nabla v \|_{L^2(\Omega)}^{2}} \quad 
{\rm for \ all} \quad u\in\overline{\mathcal{B}}_1.
\]
In other words, one has
\[
{\rm def}_{\alpha,\Omega}(u) \geq C^*{\rm dist}_{\alpha,\Omega}(u)^{2}.
\]
Moreover, there exists an open dense set of parameters $(\alpha,\Omega)$ for which 
all optimizers in $\mathcal{M}_\alpha(\Omega)\cap\overline{\mathcal{B}}_1$ are nondegenerate.
\end{theorem}

Our last main result is a quantitative stability estimate for the full range 
$0< \alpha \leq 4\pi$ in the case that 
\[
\Omega = \mathbf{D} = \{ x \in \mathbf{R}^2 : |x| < 1 \}.
\]  
\begin{theorem}[Symmetric critical case] \label{thm:critical_stability}
Let $\mathbf{D} \subset \R^2$ be the open $($unit$)$ disk and $\alpha \in (0,4\pi]$.
There exists a constant $C_\circ=C_\circ(\alpha)>0$ satisfying
\begin{equation*} 
{\rm TM}({\alpha},\mathbf{D}) - \mathscr{E}_{\alpha, \mathbf{D}} (u) \geq  C_\circ \inf_{v \in
\mathcal{M}_\alpha(\mathbf{D})} {\|\nabla u - \nabla v \|_{L^2(\mathbf{D})}^2} \quad 
{\rm for \ all} \quad u\in \overline{\mathcal{B}}_1.
\end{equation*} 
In other words, one has
\[
{\rm def}_{\alpha,\mathbf{D}}(u)\gtrsim_{\alpha}{\rm dist}_{\alpha,\mathbf{D}}(u)^2.
\]
\end{theorem} 

\subsection{Comparison to other results} Historically, stability questions first arose in 
the classical Sobolev inequality. Brezis and Lieb \cite{MR790771} first asked if one can
quantify the Sobolev inequality, and Bianchi and Egnell \cite{BE} proved this 
is possible. Since then, many others have investigated the stability, both 
qualitative and quantitative, of a variety of functional inequalities, including 
higher-order Sobolev inequalities \cite{LW,CFW}, the Gagliardo--Nirenberg--Sobolev
and log Hardy--Littlewood--Sobolev inequalities \cite{CarlenFigalli2013}, $L^p$-Sobolev inequalities for
$p> 2$ \cite{FigalliZhang2022},
the Caffarelli--Kohn--Nirenberg inequality \cite{Wei_Wu, Frank_Peteranderl} and
the Hardy--Littlewood--Sobolev inequality \cite{CLT}. Dolbeault, Esteban, Figalli, Frank, and
Loss \cite{DEFFL2025} recently uncovered the sharp dimension dependence of the 
stability constants in the Sobolev and log-Sobolev inequalities. Engelstein, Neumayer,
and Spolaor \cite{MR4427104} extended the Bianchi--Egnell result to a more geometric
setting, proving a quantitative stability inequality for the conformal
Yamabe constant; see \cite{andrade2024quantitativestabilitytotalqcurvature} for
an extension of the Engelstein, Neumayer, and Spolaor result
to higher-order curvatures. Other related work investigates the best constants in
the Bianchi--Egnell inequality \cite{MR4623703} and the attainability of the stability 
inequality \cite{arxiv:2211.14185}. This list of recent results is far from complete; 
the reader can find a nice overview of this topic in \cite{FrankKonigTang2023}. 

Most of this previous work is quite different from ours, because in 
most cases the set of optimizers is well-known, which greatly simplifies 
the investigation of the behavior of the corresponding functional on 
nearly-optimal functions. As mentioned earlier, the precise form of the
Trudinger--Moser extremals remain unknown, and so we must use substantially different techniques. 

\subsection{Outline of the proofs} 
Our methods rely on several spectral gap inequalities, and so to 
begin we state the relevant eigenvalue problems. Let $\Omega \subset 
\R^2$ be an open, bounded domain with smooth boundary, and let 
$\alpha \in (0, 4\pi]$. The first relevant eigenvalue problem is the 
well-studied \emph{first Dirichlet eigenvalue}
\begin{equation}\label{eq:dirichlet_eigenvalue}
    \mu_1(\Omega) = \inf_{\psi \in W^{1,2}_0(\Omega)}\frac{\int_\Omega |\nabla \psi|^2 \ud x}
    {\int_\Omega \psi^2 \ud x },
\end{equation}
which is associated with the problem
\begin{equation}\label{eq:limiting_BVP}\tag{$\mathcal{P}_{0,\Omega}$}
		\begin{cases}
			-\Delta \psi  = \mu  \psi  & \mbox{in} \quad\Omega\\
			\psi=0 & \mbox{on} \quad\partial \Omega.
		\end{cases}    
\end{equation}

Now let $u_* \in \mathcal{M}_\alpha (\Omega) \cap \mathcal{S}_1$ be an 
optimizer of \eqref{moser_tru_ineq}, so that $\mathscr{E}_{\alpha, \Omega}(u_*) 
= {\rm TM}(\alpha, \Omega)$ and $\| \nabla u_* \|_{L^2(\Omega)} = 1$. We define
the first \emph{$\alpha$-regularized eigenvalue}
\begin{equation}\label{eq:alpha_eigenvalue}
\lambda_1(\alpha,\Omega) = \inf_{\phi \in W^{1,2}_0(\Omega)}\frac{\int_\Omega |\nabla \phi|^2 \ud x}
{\int_\Omega \phi^2(1+2\alpha u_*^2)e^{\alpha u_*^2} \ud x},
\end{equation}
associated with the linearized  problem 
\begin{equation}\label{eq:alpha_BVP}\tag{$\mathcal{P}_{\alpha,\Omega}$}
        \begin{cases}
			-\Delta \phi  = \lambda (1 + 2\alpha u_{*}^2) e^{\alpha u_{*}^2} \phi & \mbox{in} \quad\Omega\\
			\phi=0 & \mbox{on} \quad\partial \Omega.
		\end{cases}      
\end{equation}
We also define the (nonlinear) \emph{first Trudinger--Moser eigenvalue} as
\begin{equation}\label{eq:TM_eigenvalue}
    \Lambda_1(\alpha, \Omega) = \inf_{v \in W^{1,2}_0(\Omega)} 
    \frac{\int_\Omega |\nabla v|^2 \ud x}{\int_{\Omega}v^2e^{\alpha u_*^2}\ud x},
\end{equation}
associated with the nonlinear eigenvalue problem
\begin{equation}\label{eq:TM_BVP}\tag{$\mathcal{TM}_{\alpha,\Omega}$}
		 \begin{cases}
			-\Delta v = \Lambda  e^{\alpha u_{*}^2} v & \mbox{in} \quad \Omega\\
			v=0 & \mbox{on} \quad\partial \Omega.
		\end{cases} 
	\end{equation}

We denote the eigenvalues of \eqref{eq:alpha_BVP} and \eqref{eq:TM_BVP} by
\[
\{\lambda_j(\alpha,\Omega)\}_{j\in\mathbf{N}} \quad {\rm and} \quad \{\Lambda_j(\alpha,\Omega)\}_{j\in\mathbf{N}},
\]
respectively.
For simplicity, we omit the domain in the notation when no confusion can arise.
Also, when $\Omega=\mathbf{D}$ is the unit disk, we simply denote
\begin{equation}\label{eq:eigenvalue_linearized}
 \lambda_j(\alpha, \mathbf{D})=\lambda_j^\circ(\alpha) \quad {\rm and} 
 \quad \Lambda_j(\alpha,\mathbf{D})=\Lambda_j^\circ(\alpha) \quad {\rm for} \quad j\in\mathbf{N}.   
\end{equation}
In Lemma~\ref{lm:eigenvalue_convergence} we will show that 
\[\lim_{\alpha\to0}\lambda_1(\alpha,\Omega)=\lim_{\alpha\to0}\Lambda_1(\alpha,\Omega)=\mu_1(\Omega).
\]

The proof of Theorem~\ref{thm:subcritical_stability} 
proceeds by contradiction. We suppose no uniform coercivity constant 
exists, and so there is a violating sequence $\{ u_k\}_{k \in \mathbf{N}} \subset 
\mathcal{S}_1$ such that 
\begin{equation}\label{contra0}\tag{$\mathcal{H}_0$}
			\lim_{k\rightarrow \infty}\frac{ {\rm TM}(\alpha,\Omega) - 
            \mathscr{E}_{\alpha,\Omega}(u_k)}{\inf_{v \in \mathcal{M}_\alpha(\Omega)} 
            \|{\nabla (v-u_k)}\|_{L^2(\Omega)}^2}=\lim_{k\to\infty}
            \frac{{\rm def}_{\alpha,\Omega}(u_k)}{{\rm dist}_{\alpha,\Omega}(u_k)^2}=0.
	\end{equation}
Thus, the compactness of maximizing sequences for $\mathscr{E}_{\alpha,\Omega}$ (via Vitali’s 
theorem and concentration--compactness) ensures that, up to subsequences, 
$u_k\to u_*\in\mathcal{M}_\alpha(\Omega)$ converges strongly in $W^{1,2}_0(\Omega)$. 
In addition, the normalized deviations $\{\widehat{w}_k\}_{k\in\mathbf{N}}
\subset W^{1,2}_0(\Omega)$, defined as
\begin{equation}\label{eq:normalized_deviations}
\widehat{w}_k=\frac{u_k-u_*}{\|\nabla(u_k-u_*)\|_{L^2(\Omega)}} \quad {\rm for} \quad k\in\mathbf{N},
\end{equation}
satisfy an asymptotic Euler--Lagrange equation, and the limit 
$\widehat{w}_*:=\lim_{k\to\infty}\widehat{w}_k$ solves the linearized eigenvalue 
problem \eqref{eq:alpha_BVP}. We then use Lemma~\ref{lm:spectral_gap} to 
conclude that when $\alpha$ is sufficiently small this cannot occur. 

The proof of Theorem~\ref{thm:subcritical_stability_nondeg} is analogous to the works of 
Bianchi and Egnell~\cite{BE} for the Sobolev inequality, 
and Engelstein, Neumayer, and Spolaor~\cite{MR4427104} for the Yamabe problem. 
In this situation, the quantitative exponent in our stability estimate depends on the spectral properties 
of the second variation at the family of maximizers~$\mathcal{M}_\alpha(\Omega)$.
When all critical points are nondegenerate or integrable, a local Lyapunov--Schmidt reduction 
implies \emph{quadratic} coercivity, hence the optimal quadratic exponent.
Moreover, such nondegeneracy holds \emph{generically}: for an open dense set of parameters 
$(\alpha,\Omega)$, all maximizers are nondegenerate.

The proof of Theorem~\ref{thm:critical_stability} is slightly different and exploits the symmetry of the domain.
On the unit disk $\mathbf{D}$, symmetry and compactness can be restored without smallness in $\alpha$. 
Initially, the results of Carleson and Chang \cite{Carleson-Chang} and Gidas, Ni, and 
Nirenberg  \cite{MR634248} ensure that every extremal $\bar{u}_*\in\mathcal{M}_\alpha(\mathbf{D})$ 
is positive, radial, and strictly decreasing. 
Furthermore, the linearized operator around $u_*$ separates into angular modes, and the $j=1$ 
nonradial sector introduces a positive centrifugal term $r^{-2}$. 
This produces a strict mode gap $\Lambda_1^\circ(\alpha)<\lambda_2^\circ(\alpha)$ between the radial 
eigenvalue of \eqref{eq:TM_eigenvalue} and the first nonradial eigenvalues of \eqref{eq:alpha_BVP}, 
ensuring that the quadratic form associated with the second variation is positive on the entire 
tangent space $T_{u_*}\mathcal{S}_1$. 
The same contradiction argument as in the subcritical case then shows that any sequence violating 
stability would lead to a sign-definite limit eigenfunction, which is again impossible by orthogonality, 
thereby providing a global stability estimate for all $\alpha\in(0,4\pi]$.
\medskip 

Let us describe the plan for the rest of the manuscript.
\S~\ref{sec:preliminaries} introduces the functional setting and collects basic
properties of the Trudinger--Moser functional, including the asymptotic comparison of
the associated eigenvalues and the statement of the Krein--Rutman theorem, which plays
a central role in the spectral analysis.
In \S~\ref{sec:subcritical_regime} we establish the quantitative stability in the
subcritical regime. 
The argument first derives a local coercivity estimate on the unit sphere
via a contradiction strategy based on the sign of the limiting
eigenfunction, and then extends the inequality to the closed unit ball
using the monotonicity of the Trudinger--Moser functional
along radial rescalings.
\S~\ref{sec:critical_regime} is devoted to the critical case on the unit disk,
where the rotational symmetry allows a decomposition into angular modes and the proof
of a strict spectral gap between the first radial and nonradial eigenvalues, leading
to a direct stability estimate without invoking the contradiction argument.
\smallskip

\noindent {\sc Acknowledgement:} Much of this research took place while A.C.M. visited 
the Julius-Maximilians-Universit\"at W\"urzburg from December 2024 to December 2025. 
He thanks the Institute 
of Mathematics for its hospitality. Part of this work was also discussed during 
visits of J.H.A. to the Julius-Maximilians-Universit\"at W\"urzburg on the occasion of 
the Geometric Analysis Workshops in 2025 and 2026, and of J.M.~do~\'O during the 2026 
edition; they gratefully acknowledge the support and hospitality of the Institute.

\section{Preliminaries}\label{sec:preliminaries}
This section collects some preliminary definitions and results that will be used throughout the manuscript.

\subsection{The Krein--Rutman principle}
The positivity and simplicity of the first eigenvalue for the linearized problems
\eqref{eq:alpha_BVP} and \eqref{eq:TM_BVP} follow from the Krein--Rutman theorem,
which extends the Perron--Frobenius theory to compact, strongly positive operators
in ordered Banach spaces.  We recall a convenient formulation for our purposes.

For further background on positivity methods in nonlinear elliptic problems, we 
refer to de~Figueiredo~\cite[Proposition~1.12A, Theorem~1.13]{deFigueiredo1982}.

\begin{theoremletter}[Krein--Rutman \cite{MR2759829}]\label{thm:krein-rutman}
Let $X$ be a real Banach space ordered by a closed, convex, solid cone $\mathcal{K}\subset X$ with 
nonempty interior ${\rm int}(\mathcal{K})\neq \varnothing$ and $\mathscr{T}:X\to X$ be a compact linear 
operator.
If $\mathscr{T}(\mathcal{K}\setminus\{0\})\subset {\rm int}(\mathcal{K})$, then, the spectral 
radius $\rho(\mathscr{T})>0$ is a simple eigenvalue of $\mathscr{T}$ with an eigenvector 
$\phi_1\in {\rm int}(\mathcal{K})$ (strictly positive in the cone ordering),
and no other eigenvalue admits a nonnegative eigenvector.
Moreover, if $\mathcal{S}\subset X$ defined as $\mathcal{S}:=\{u\in X :\langle u,\phi\rangle=0\}$ 
is the hyperplane orthogonal to any
continuous positive linear functional $\phi\in \mathcal{K}^*$, then
\[
\left\|\mathscr{T}^k\big|_{\mathcal{S}}\right\|=\mathrm{o}(\rho(\mathscr{T})^k)\quad as \quad k\to\infty,
\]
and so $\rho(\mathscr{T})\in\mathbf{R}$ is algebraically and geometrically simple and isolated in the spectrum.
\end{theoremletter}

\begin{remark}
In our setting, $X=L^2(\Omega)$, $\mathcal{K}=\{u\in W^{1,2}_0(\Omega) : u\ge0\}$, and
$\mathscr{T}:L^2(\Omega)\to L^2(\Omega)$ is the Green operator associated with
\eqref{eq:alpha_BVP} or \eqref{eq:TM_BVP}.  
The Krein--Rutman theorem ensures that
both smallest eigenvalues $\lambda_1(\alpha,\Omega)$ and $\Lambda_1(\alpha,\Omega)$
are positive,  simple, $\lambda_1(\alpha,\Omega)< \Lambda_1(\alpha,\Omega)$ and possess 
a strictly positive eigenfunction.
\end{remark}

\subsection{Compactness of optimizing sequences}
Now, we will prove that any maximizing sequence for the Trudinger--Moser constant shall, up to a 
subsequence, converge strongly to some extremal.
	\begin{lemma}[Integral convergence]\label{lm:vitali_convergence}
        Let $\Omega\subset\mathbf{R}^2$ be a smooth, bounded domain and $\alpha\in (0,4\pi]$.
       Let $\{u_k\}_{k\in\mathbf{N}} \subset W^{1,2}_0(\Omega)$ be a sequence and suppose that 
       there exist $u_*\in W^{1,2}_0(\Omega)$ and $q>1$ such that $u_k \rightarrow u_*$ a.e. in $\Omega$ and
        \begin{equation}\label{eq:uniform_integrability}
         \sup_{k\in\mathbf{N}}\int_{\Omega} e^{\alpha q {|u_k|}^{2}} \ud x <\infty.
        \end{equation}
		Then one has
		\begin{equation}\label{eq:vitali_convergence}
			\lim_{k\rightarrow \infty}\int_{\Omega} e^{\alpha  {|u_k|}^{2}} 
            \ud x= \int_{\Omega} e^{\alpha  {|u_*|}^{2}} \ud x.
		\end{equation}	
	\end{lemma}

    \begin{proof} Setting $f_k=e^{\alpha |u_k|^2}$ and $f=e^{\alpha  {|u_*|}^{2}}$ we have 
that $f_k \to f$ a.e.\ in $\Omega$ and by \eqref{eq:uniform_integrability} we get that 
$\{f_k\}_{k\in\N}$ is uniformly bounded in $L^q (\Omega)$. In addition, the Trudinger--Moser 
inequality \eqref{moser_tru_ineq} implies that $f\in L^1(\Omega)$, and 
so \eqref{eq:vitali_convergence} follows from \cite[Lemma~3.2]{LuLu2020}. 
    \end{proof}    
	\begin{lemma}[Compactness of maximizing sequences]\label{lm:compactness}
		Let $\Omega\subset\mathbf{R}^2$ be a smooth, bounded domain, and $\alpha\in (0,4\pi]$.
        If $\{u_k\}_{k\in\mathbf{N}} \subset  \overline{\mathcal{B}}_1$ is a maximizing sequence 
        for the Trudinger--Moser inequality on the closed unit ball, that is, 
        \[
        \lim_{k\rightarrow \infty}\int_{\Omega} e^{\alpha  {|u_k|}^{2}} \ud x={\rm TM}({\alpha},\Omega).
        \]
        Then, there exists an optimizer $u_*\in \mathcal{M}_\alpha(\Omega)\cap \mathcal{S}_1$ such that, 
        up to a subsequence, $u_k \rightarrow u_*$ strongly in $W^{1,2}_0(\Omega)$ as $k\to\infty$. 
	\end{lemma}
    
	\begin{proof} First, we note that 
    \begin{equation}\label{maxB}
         u_*\in \mathcal{M}_\alpha(\Omega)\cap \overline{\mathcal{B}}_1\quad {\rm implies} 
         \quad u_*\in\mathcal{M}_\alpha(\Omega)\cap\mathcal{S}_1.
    \end{equation}
    In fact, if $u_*\in \mathcal{M}_\alpha(\Omega)\cap \overline{\mathcal{B}}_1$ and 
    $0<\|\nabla u_*\|_{L^2(\Omega)}<1$, we obtain $\mathscr{E}_{\alpha,\Omega}(v)>
    \mathscr{E}_{\alpha,\Omega}(u_*)={\rm TM}(\alpha,\Omega)$, where $v =\tfrac{u_*}{\theta}$ 
    with $\theta=\|\nabla u_*\|_{L^2(\Omega)}$. Since $v\in\overline{\mathcal{B}}_1$, this 
    yields a contradiction. Hence, $\theta=\|\nabla u_*\|_{L^2(\Omega)}=1$, and so \eqref{maxB} holds.

Since $\{u_k\}_{k\in\mathbf{N}} \subset \overline{\mathcal{B}}_1$ is uniformly bounded in a Hilbert space, 
there exists $u_*\in W^{1,2}_0(\Omega)$ such that, up to a subsequence, $u_k \rightharpoonup u_*$ weakly 
in $W^{1,2}_0(\Omega)$. By the weak lower semicontinuity of the norm, we can write 
 \begin{equation}\label{semi-continuity}
            0\leq\|\nabla u_*\|_{L^2(\Omega)}\leq \liminf_{k\to\infty} 
            \|\nabla u_k\|_{L^2(\Omega)}\leq \limsup_{k\to\infty} \|\nabla u_k\|_{L^2(\Omega)}\leq 1,
            \end{equation}
            and consequently 
        \begin{equation}\label{eq:u_*inball}
            u_*\in\overline{\mathcal{B}}_1 .
        \end{equation}    
Further, from the compact embedding $W^{1,2}_0(\Omega)\hookrightarrow L^p(\Omega),\; p>1$, up to a subsequence, 
\begin{equation}\label{Lp-pointwise-convergence}
  u_k \rightharpoonup u_*\;\;\mbox{weakly in}\;\; W^{1,2}_0(\Omega),\quad u_k\to u_*\;\; 
  \mbox{in }\;\; L^p(\Omega)\;\; \mbox{and}\;\; u_k\to u_*\;\; \mbox{a.e.\ in }\;\;\Omega.
 \end{equation}
        We divide the rest of the proof into several claims as follows:

        \noindent{\bf Claim 1:} If $\alpha\in(0,4\pi)$, then $u_*\in \mathcal{M}_\alpha(\Omega)
        \cap \mathcal{S}_1$.
        
 		\noindent In fact, by choosing $q>1$ (close to $1$) such that $q\alpha\le 4\pi$, we 
        have $\mathscr{E}_{q\alpha, \Omega} (u_k)\le  \mathscr{E}_{4\pi, \Omega} (u_k)$ and the 
        estimate \eqref{moser_tru_ineq} implies that \eqref{eq:uniform_integrability} holds. 
        Thus, from \eqref{Lp-pointwise-convergence} we can apply  Lemma~\ref{lm:vitali_convergence} to 
        obtain
		\begin{equation}\label{conv}
			\lim_{k\to\infty}\int_{\Omega} e^{\alpha {|u_{k}|}^{2}}\ud x = \int_{\Omega} 
            e^{\alpha {|u_*|}^{2}}\ud x={\rm TM}(\alpha,\Omega).
		\end{equation}
       By combining \eqref{eq:u_*inball} with \eqref{conv}, we have $u_*\in
       \mathcal{M}_\alpha(\Omega)\cap\overline{\mathcal{B}}_1.$ From \eqref{maxB}, we 
       obtain $u_*\in \mathcal{M}_\alpha(\Omega)\cap \mathcal{S}_1$.
        
        \noindent{\bf Claim 2:} If $\alpha=4\pi$, then $u_*\in \mathcal{M}_\alpha(\Omega)\cap \mathcal{S}_1$.

    \noindent The critical case $\alpha=4\pi$ is more involved. We shall use the 
    concentration--compactness alternative due to Lions \cite{Lions85} (see also \cite{Cerny2013}). First, 
    let us denote by $\mathbb{M}(\overline{\Omega})$ the space of Radon measures on $\overline{\Omega}$.  
    As proved in \cite[Sec.~3, page~115]{MR2759829}, the boundedness of $\{u_k\}_{k\in\mathbf{N}}$ 
    in $W^{1,2}_0(\Omega)$ implies that, up to a subsequence, 
    \begin{equation}\label{convergence-Radon}
        |\nabla u_k|^2\ud x\overset{*}{\rightharpoonup} \mu \quad \text{in the weak* topology of  } 
        \mathbb{M}(\overline{\Omega}).
    \end{equation}
    By using \eqref{Lp-pointwise-convergence} and \eqref{convergence-Radon}, up to a subsequence,  we obtain
    \begin{equation}\label{Cerny-H}
        u_k \rightharpoonup u_*\;\; \text{in} \;\; W^{1,2}_0(\Omega),\quad u_k\to u_*\;\; 
        \text{a.e.\ in} \;\;\Omega\;\; \text{and} \;\;  |\nabla u_k|^2\ud x\overset{*}{\rightharpoonup} 
        \mu \;\; \text{in} \;\;\mathbb{M}(\overline{\Omega}).
    \end{equation}
   Hence, according to Lions' concentration--compactness alternative, there are only three possibilities:
   \begin{itemize}
        \item [(i)] If $u_*\equiv0$ and $\mu=\delta_{x_0}$ is the Dirac mass concentrated measure at 
        some point $x_0\in\overline{\Omega}$, then
        \begin{equation*}
            \lim_{k\to\infty}\int_{\Omega}e^{4\pi {|u_{k}|}^{2}} \ud x=c+|\Omega|
        \end{equation*}
        for some $c\in [0, \infty)$.
        \item [(ii)] If $u_*\equiv 0$ and $\mu$ is not a Dirac mass concentrated at one point, 
        then there is $q>1$ such that 
        \begin{equation*}
             \sup_{k\in\mathbf{N}}\int_{\Omega}e^{4\pi q{|u_{k}|}^{2}} \ud x<\infty.
        \end{equation*}
        \item [(iii)] If $u_*\not\equiv 0$, then there is $q>1$ such that
        \begin{equation*}
             \sup_{k\in\mathbf{N}}\int_{\Omega}e^{4\pi q{|u_{k}|}^{2}} \ud x<\infty.
        \end{equation*}
    \end{itemize}
   Notice that Lemma~\ref{lm:vitali_convergence} implies, for cases (ii) and (iii), that  
    \begin{equation}\label{Lions-compact}
         {\rm TM}(4\pi, \Omega)=\lim_{k\rightarrow \infty}\int_{\Omega} e^{4\pi  {|u_k|}^{2}} \ud x=\int_{\Omega} e^{4\pi {|u_*|}^{2}}\ud x.
    \end{equation}
    Since ${\rm TM}(4\pi, \Omega)>|\Omega|$, \eqref{Lions-compact} excludes the possibility (ii). In order to exclude possibility (i), we recall that 
 Flucher \cite{Flucher1992} (see also \cite{deFigueiredo2002,Lin1996}) proved that the concentration level is strictly lower than ${\rm TM}(4\pi, \Omega)$, namely,
\[{\rm TM}(4\pi, \Omega)>\sup_{x\in\overline{\Omega}} C_{\Omega}(x),\]
where
\begin{equation*}
  C_{\Omega}(x)
=
\sup\left\{
\limsup_{k\to\infty} \int_{\Omega} e^{4\pi |u_k|^2}\,\ud x
:
\|\nabla u_k\|_{L^2(\Omega)}\le 1,\;\;
|\nabla u_k|^2\ud x\overset{*}{\rightharpoonup} \delta_x \;\; \text{in} \;\;\mathbb{M}(\overline{\Omega}) \;\; \text{with} \;\; x\in \overline{\Omega}
\right\}. 
\end{equation*} 
Since $\{u_k\}_{k\in\N}$ is a maximizing sequence, the concentration case (i) does not occur. Consequently, possibility (iii) must occur, and \eqref{Lions-compact} holds as well. From \eqref{maxB}, \eqref{eq:u_*inball}, and \eqref{Lions-compact}, we must have $u_*\in \mathcal{M}_{4\pi}(\Omega)\cap \mathcal{S}_1$.

        \noindent{\bf Claim 3:} If $\alpha\in (0,4\pi]$, then $u_*\in\mathcal{M}_\alpha(\Omega)\cap \mathcal{S}_1$ and $\lim_{k\to\infty}\|\nabla (u_k -u_*)\|_{L^{2}(\Omega)}=0$.
        
 		\noindent From  Claims 1 and 2, we already have $u_*\in\mathcal{M}_\alpha(\Omega)\cap \mathcal{S}_1$, for any $\alpha\in (0,4\pi]$.  In addition, from \eqref{semi-continuity}, it follows that 
        \[
        \limsup_{k\to\infty} \|\nabla u_k\|_{L^2(\Omega)}\le 1=\|\nabla u_*\|_{L^2(\Omega)}.
        \]
        By virtue of the Hilbert structure of $W^{1,2}_0(\Omega)$, we can write
\begin{align*}
     0\le \| \nabla(u_k - u_*)\|^{2}_{L^{2}(\Omega)} &= \|\nabla u_k\|^2_{L^2(\Omega)}-2 \int_{\Omega}\nabla u_k\nabla u_*\, \ud x + \|\nabla u_*\|^2_{L^2{(\Omega)}}\\
     &\le 2\|\nabla u_*\|^2_{L^2{(\Omega)}}-2\int_{\Omega}\nabla u_k\nabla u_*\, \ud x \to 0 \quad {\rm as}\quad k\to \infty
\end{align*}
due to the weak convergence  $u_k \rightharpoonup u_*$ in $W^{1,2}_0(\Omega)$. Thus, $u_k \rightarrow u_*$ strongly in $W^{1,2}_0(\Omega)$ as $k\to\infty$. 
        The proof of the lemma is a consequence of the last claim.
	\end{proof}

\begin{corollary}[Qualitative stability]\label{maxseqconv}
    Let $\Omega\subset\mathbf{R}^2$ be a smooth, bounded domain and $\alpha\in (0,4\pi]$.
    If $\{u_k\} \subset \overline{\mathcal{B}}_1$ is a maximizing sequence for the Trudinger--Moser supremum ${\rm TM}(\alpha,\Omega)$, then it satisfies
    \[
    \lim_{k\rightarrow \infty}\inf_{v \in \mathcal{M}_\alpha(\Omega)\cap \mathcal{S}_1} \|\nabla (v-u_k)\|_{L^2(\Omega)}^2=0.
    \]
\end{corollary}

\begin{proof}
    Suppose by contradiction that the statement is false. Then, there exist a subsequence $\{u_{k_\ell}\}_{\ell\in\mathbf{N}}\subset \overline{\mathcal{B}}_1$ and a constant $\delta > 0$ such that $\inf_{v\in\mathcal{M}_\alpha(\Omega)\cap\mathcal{S}_1}\|\nabla(v-u_{k_\ell})\|_{L^2(\Omega)} > \delta$ for all $\ell\in\mathbf{N}$. Passing to a further subsequence, by Lemma~\ref{lm:compactness}, $u_{k_\ell} \rightarrow u$ strongly in $W_0^{1,2}(\Omega)$ for some $u\in\mathcal{M}_\alpha(\Omega)\cap\mathcal{S}_1$, which contradicts $\|\nabla(u-u_{k_\ell})\|_{L^2(\Omega)}>\delta$.
\end{proof}

\begin{corollary}[Compactness of optimizing sequences]\label{co:compactness}
		Let $\Omega\subset\mathbf{R}^2$ be a smooth, bounded domain and $\alpha\in (0,4\pi]$.
        If $\{u_k\}_{k\in\mathbf{N}} \subset \overline{\mathcal{B}}_1$ is an optimizing sequence for ${\rm TM}(\alpha,\Omega)$  then there exists an optimizer $u_*\in \mathcal{M}_\alpha(\Omega)\cap \mathcal{S}_1$ such that, up to a subsequence, $u_k \rightarrow u_*$ strongly in $W^{1,2}_0(\Omega)$ as $k\to\infty$.
	\end{corollary}

\subsection{From the unit sphere to the closed unit ball}
The quantitative stability results established so far concern functions normalized
on the unit sphere.
In contrast with the Sobolev functional, which is homogeneous and hence scale-invariant,
the Trudinger--Moser functional does not preserve its value under rescalings of the Dirichlet energy.
Thus, to obtain a formulation valid for arbitrary functions satisfying
$\|\nabla u\|_{L^2(\Omega)}\le1$, one must extend the stability inequality from the sphere to the closed unit ball, which is the content of the next lemma.

Before establishing the extension result, let us set some preliminary notation.
For any $u\in\overline{\mathcal{B}}_1\setminus\{0\}$, we set 
$\theta:=\|\nabla u\|_{L^2(\Omega)}\in(0,1]$.
We define the normalized projection $\Pi:\overline{\mathcal{B}}_1\setminus\{0\}\to \mathcal{S}_1$ as 
\(\Pi(u):=\frac{u}{\theta}\).

\begin{lemma}[Extension from the sphere to the ball]
\label{lm:sphere_to_ball}
Let $\Omega\subset\mathbf{R}^2$ be a smooth bounded domain, $\alpha\in(0,4\pi]$, and $\zeta\geq 2$.
Assume that
\begin{equation}\label{eq:stability_sphere}
{\rm def}_{\alpha, \Omega} (u) \gtrsim_{\alpha, \Omega} {\rm dist}_{\alpha, \Omega}(u)^\zeta
\quad {\rm for \ all}\quad u\in \mathcal{S}_1.
\end{equation}
Then one has
\begin{equation}\label{eq:stability_ball}
{\rm def}_{\alpha,\Omega}(u)\gtrsim_{\alpha,\Omega,\zeta}{\rm dist}_{\alpha,\Omega}(u)^\zeta \quad {\rm for \ all} \quad u\in  \overline{\mathcal{B}}_1.
\end{equation}
\end{lemma}
\begin{proof} We give the proof in the case $\zeta=2$; the proof for $\zeta>2$ is 
very similar. 
Suppose by contradiction that there exists a sequence $\{\widetilde{u}_k\}_{k\in\mathbf{N}}
\subset {\rm int}(\mathcal{B}_1)$, {\it i.e.,} $0<\|\nabla \widetilde{u}_k\|_{L^2(\Omega)}<1$ and 
\begin{equation}\label{contradsphereball}
    \lim_{k\rightarrow \infty}\frac{{\rm TM}({\alpha},\Omega) - \mathscr{E}_{\alpha, \Omega} (\widetilde{u}_k)}{{\rm dist}_{\alpha,\Omega}^2(\widetilde{u}_k)}=0.
\end{equation}
In addition, notice that $\widetilde{u}_k\in W^{1,2}_0(\Omega)$ must be a maximizing sequence  for the Trudinger--Moser inequality, that is, $\lim_{k\to\infty}\mathscr{E}_{\alpha,\Omega}(\widetilde{u}_k)={\rm TM}({\alpha},\Omega)$. 
Passing to a subsequence if necessary, we may use Lemma~\ref{lm:compactness} and assume that $\widetilde{u}_k\to \widetilde{u}_*$ converges strongly in $W^{1,2}_0(\Omega)$ for some $\widetilde{u}_*\in \mathcal{M}_{\alpha}(\Omega)\cap\mathcal{S}_1$ and also
\[
-\infty<\lim_{k \rightarrow \infty}\frac{{\rm dist}_{\alpha,\Omega}(\Pi(\widetilde{u}_k))}{{\rm dist}_{\alpha,\Omega}(\widetilde{u}_k)}<\infty.
\]
Since ${\rm dist}_{\alpha,\Omega}(\widetilde{u}_k)\geq  {\rm dist}_{\alpha,\Omega}\left(\Pi(\widetilde{u}_k)\right)$ for $k\gg1$ large, it is not hard to check that $\Pi(\widetilde{u}_k)\to u_*$ must also converge strongly in $W^{1,2}_0(\Omega)$.

We distinguish two cases:

\noindent{\it Case 1:} There exists $\delta>0$ such that $\lim_{k \rightarrow \infty}\frac{{\rm dist}_{\alpha,\Omega}(\Pi(\widetilde{u}_k))}{{\rm dist}_{\alpha,\Omega}(\widetilde{u}_k)}>\delta>0$.

\noindent By means of \eqref{eq:stability_sphere}, we have the uniform estimate below:
\begin{align*}
    \frac{{\rm TM}({\alpha},\Omega) - \mathscr{E}_{\alpha, \Omega} (\widetilde{u}_k)}{{\rm dist}_{\alpha,\Omega}^2(\widetilde{u}_k)}&\geq \frac{{\rm TM}({\alpha},\Omega) - \mathscr{E}_{\alpha, \Omega} (\Pi(\widetilde{u}_k))}{{\rm dist}_{\alpha,\Omega}^2(\widetilde{u}_k)}\geq \delta^2\frac{{\rm TM}({\alpha},\Omega) - \mathscr{E}_{\alpha, \Omega} (\Pi(\widetilde{u}_k))}{{\rm dist}_{\alpha,\Omega}^2(\Pi(\widetilde{u}_k))}\gtrsim \delta^2 \quad {\rm for} \quad k\gg1,
\end{align*}
which contradicts \eqref{contradsphereball}.

\noindent{\it Case 2:} One has $\lim_{k \rightarrow \infty}\frac{{\rm dist}_{\alpha,\Omega}(\Pi(\widetilde{u}_k))}{{\rm dist}_{\alpha,\Omega}(\widetilde{u}_k)}=0$.

\noindent In this case, we observe 
\[
{\rm dist}_{\alpha,\Omega}(\widetilde{u}_k)\leq {\rm dist}_{\alpha,\Omega}\left(\Pi(\widetilde{u}_k)\right)+\left(1-\|\nabla \widetilde{u}_k\|_{L^2(\Omega)}\right)\leq 3 {\rm dist}_{\alpha,\Omega}(\widetilde{u}_k).
\]
From this, we can find $k\gg1$ sufficiently large satisfying
\[
\frac{1}{2}< \frac{\left(1-\|\nabla \widetilde{u}_k\|_{L^2(\Omega)}\right)}{{\rm dist}_{\alpha,\Omega}(\widetilde{u}_k)}.
\]
Consequently, we get
\begin{align*}
    \frac{{\rm TM}({\alpha},\Omega) - \mathscr{E}_{\alpha, \Omega} (\widetilde{u}_k)}{{\rm dist}_{\alpha,\Omega}^2(\widetilde{u}_k)}&=\frac{{\rm TM}({\alpha},\Omega) - \mathscr{E}_{\alpha, \Omega} (\Pi(\widetilde{u}_k))}{{\rm dist}_{\alpha,\Omega}^2(\widetilde{u}_k)}+\frac{\mathscr{E}_{\alpha, \Omega} \left( \Pi(\widetilde{u}_k)\right) - \mathscr{E}_{\alpha, \Omega} (\widetilde{u}_k)}{{\rm dist}_{\alpha,\Omega}^2(\widetilde{u}_k)}\\
    &\geq \frac{\mathscr{E}_{\alpha, \Omega} \left( \Pi(\widetilde{u}_k)\right) - \mathscr{E}_{\alpha, \Omega} (\widetilde{u}_k)}{{\rm dist}_{\alpha,\Omega}^2(\widetilde{u}_k)}\\
    &> \frac{1}{4}\left[\frac{\mathscr{E}_{\alpha, \Omega} \left( \Pi(\widetilde{u}_k)\right) - \mathscr{E}_{\alpha, \Omega} (\widetilde{u}_k)}{\left(1-\|\nabla \widetilde{u}_k\|_{L^2(\Omega)}\right)^2}\right]\\
    &\ge \frac{\alpha}{4}\left[\frac{1-\|\nabla \widetilde{u}_k\|_{L^2(\Omega)}^2}{\left(1-\|\nabla \widetilde{u}_k\|_{L^2(\Omega)}\right)^2}\right]\|\Pi(\widetilde{u}_k)\|_{L^2(\Omega)}^2,
\end{align*}
and so
\[
\frac{{\rm TM}({\alpha},\Omega) - \mathscr{E}_{\alpha, \Omega} (\widetilde{u}_k)}{{\rm dist}_{\alpha,\Omega}^2(\widetilde{u}_k)}\ge \frac{\alpha}{4}\left(\frac{1+\|\nabla \widetilde{u}_k\|_{L^2(\Omega)}}{1-\|\nabla \widetilde{u}_k\|_{L^2(\Omega)}}\right)\|\Pi(\widetilde{u}_k)\|_{L^2(\Omega)}^2\rightarrow \infty \quad {\rm as} \quad k \to \infty,
\]
which contradicts \eqref{contradsphereball}.
\end{proof}

\begin{remark}[Extension mechanism]
\label{rmk:sphere-to-ball}
In contrast with the Sobolev inequality, whose sharp quotient is homogeneous and
thus automatically invariant under rescaling of $\|\nabla u\|_{L^2(\Omega)}$,
the Trudinger--Moser functional lacks this homogeneity.
Nevertheless, the monotonicity of the map
$t\mapsto\mathscr{E}_{\alpha,\Omega}(t\,u)$ for $t\in[0,1]$
compensates for the absence of scale invariance and guarantees that the stability
inequality valid on the unit sphere $\mathcal{S}_1$ extends to the closed unit ball
$\overline{\mathcal{B}}_1$ without loss of generality.
\end{remark}

\section{Subcritical regime on domains}\label{sec:subcritical_regime}
This section is devoted to the proof of Theorems~\ref{thm:subcritical_stability} and \ref{thm:subcritical_stability_nondeg}.

\subsection{First eigenvalue asymptotic comparison}\label{subsec:first_eigenvalue}
In our next lemma, we analyze the asymptotic behavior of the first eigenvalues associated with the linearized problems \eqref{eq:alpha_BVP} and \eqref{eq:TM_BVP}.  
Both operators can be regarded as $\alpha$-perturbations of the Dirichlet Laplacian \eqref{eq:limiting_BVP}, with exponentially weighted potentials depending on the corresponding extremal $u_*$.  
The following result shows that, as $\alpha\to0$, these weights converge uniformly to $1$, so that the associated eigenvalues approach the first Dirichlet eigenvalue $\mu_1(\Omega)$.  
Moreover, the inequality $\lambda_1(\alpha)<\Lambda_1(\alpha)$ reflects the fact that the stronger weight $(1+2\alpha u_*^2)e^{\alpha u_*^2}$ in \eqref{eq:alpha_BVP} produces a smaller Rayleigh quotient.

    \begin{lemma}[First eigenvalue asymptotic gap]\label{lm:eigenvalue_convergence}
        Let $\Omega\subset\mathbf{R}^2$ be a smooth, bounded domain.
        Then $\lambda_1(\alpha)<\Lambda_1(\alpha)<\mu_1$
         and
         \[
            \lim_{\alpha\to 0}\lambda_1(\alpha)=\lim_{\alpha\to 0}\Lambda_1(\alpha)=\mu_1.
        \]
        In other words, given $\delta\in(0,\mu_1)$ there exists $0<\alpha_0=\alpha_0(\delta)$, 
        independent of the optimizer such that $\alpha\in(0,\alpha_0)$ implies
    \begin{equation*} 
    \mu_1-\delta < \lambda_1(\alpha) < \mu_1. 
    \end{equation*} 
    \end{lemma}

\begin{proof}
Let $u_{*}\in \mathcal{M}_\alpha(\Omega)$ be an optimizer for the Trudinger--Moser constant, that is,
\[
\mathscr{E}_{\alpha,\Omega}(u_*)={\rm TM}({\alpha},\Omega)
\quad\text{and}\quad
\|\nabla u_*\|_{L^2(\Omega)}=1.
\]
A direct computation shows that defining
\begin{equation}\label{eq:first_TM_eigenvalue}
 \Lambda_1({\alpha})=\frac{1}{\int_{\Omega} u_*^2 e^{\alpha u_*^2} \ud x}
\end{equation}
implies that $|u_*|\in \mathcal{M}_\alpha(\Omega)$ satisfies
\begin{equation}\label{eq:u*_abs}
\begin{cases}
-\Delta |u_*|=\Lambda_1({\alpha}) e^{\alpha |u_*|^2} |u_*| & \text{in} \quad\Omega,\\
|u_*|=0 & \text{on} \quad \partial \Omega.
\end{cases}
\end{equation}
By the maximum principle, $|u_*|>0$ in $\Omega$. 
Hence $u_*$ does not change sign; without loss of generality, we assume $u_*>0$. 
Then $u_{*}\in \mathcal{M}_\alpha(\Omega)$ solves the Euler--Lagrange equation associated with \eqref{eq:TM_BVP}
\begin{equation}\label{eqext}
\begin{cases}
-\Delta u_*=\Lambda_1({\alpha})  e^{\alpha u_*^2} u_* & \text{in} \quad \Omega,\\
u_*=0 & \text{on} \quad \partial \Omega,
\end{cases}
\end{equation}
where $\Lambda_1(\alpha)$ is the first Trudinger--Moser eigenvalue defined in \eqref{eq:TM_eigenvalue}. Indeed, by the Krein--Rutman principle in Theorem~\ref{thm:krein-rutman}, the first eigenvalue of \eqref{eq:TM_BVP} can be characterized as
\[
\Lambda_1({\alpha})
=\inf_{v \in W^{1,2}_0(\Omega)\setminus\{0\}}
\frac{\int_{\Omega} |\nabla v|^2 \ud x}
{\int_{\Omega} e^{\alpha u_*^2} v^2 \ud x}
\leq
\inf_{\psi \in W^{1,2}_0(\Omega)\setminus\{0\}}
\frac{\int_{\Omega} |\nabla \psi|^2 \ud x}
{\int_{\Omega}  \psi^2 \ud x}
=\mu_1(\Omega),
\]
where $\mu_1>0$ denotes the first Dirichlet eigenvalue of the Laplacian.

We now establish a uniform $L^\infty$ bound for the optimizers $u_*$ as $\alpha\to0$.  
Recall that $u_{*}\in \mathcal{M}_\alpha(\Omega)$ satisfies $\|\nabla u_*\|_{L^2(\Omega)}=1$ and solves the boundary value problem
\eqref{eqext}.
By elliptic regularity for the Dirichlet Laplacian, one has
\begin{equation}\label{eq:elliptic_estimate}
\|u_*\|_{W^{2,2}(\Omega)} \le C_\Omega\,\|\Lambda_1(\alpha)e^{\alpha u_*^2}u_*\|_{L^2(\Omega)}
= C_\Omega\,\Lambda_1(\alpha)\|e^{\alpha u_*^2}u_*\|_{L^2(\Omega)}.
\end{equation}
We next bound the right-hand side uniformly in $\alpha$.  
Since $\Lambda_1(\alpha)\le\mu_1(\Omega)$ and $u_*\in\mathcal{S}_1$, we can apply Hölder’s inequality and the Trudinger--Moser estimate to obtain
\[
\|e^{\alpha u_*^2}u_*\|_{L^2(\Omega)}^2
=\int_\Omega u_*^2 e^{2\alpha u_*^2}\ud x
\le
\|u_*\|_{L^4(\Omega)}^2
\left(\int_\Omega e^{4\alpha u_*^2}\ud x\right)^{1/2}
\le
C_\Omega
\left(\int_\Omega e^{4\alpha u_*^2}\ud x\right)^{1/2}.
\]
For $\alpha$ sufficiently small, say $\alpha\in(0,\alpha_0)$ with $\alpha_0<\pi$, the last integral is uniformly bounded by the Trudinger--Moser inequality, since $\| \nabla u_* \|_2 = 1$ and $4\alpha_0<4\pi$.  
Consequently, uniformly, we get
\[
\|e^{\alpha u_*^2}u_*\|_{L^2(\Omega)}\le C_\Omega,
\quad
{\rm for } \quad \alpha\in(0,\alpha_0),
\]
and from \eqref{eq:elliptic_estimate}, we obtain
\[
\|u_*\|_{W^{2,2}(\Omega)} \le C_\Omega'.
\]
Because $\Omega\subset\mathbf{R}^2$, the Sobolev embedding $W^{2,2}(\Omega)\hookrightarrow \mathcal{C}^{0,\gamma}(\overline{\Omega})$ (for any $\gamma\in(0,1)$) implies
\[
\|u_*\|_{L^\infty(\Omega)} \le C_\Omega''\,\|u_*\|_{W^{2,2}(\Omega)} \le M,
\]
for some constant $M>0$ depending only on $\Omega$ and independent of $\alpha$ as long as $\alpha\in(0,\alpha_0)$.  
Thus, we obtain the desired uniform bound
\begin{equation}\label{eq:uinftybound}
\|u_*\|_{L^\infty(\Omega)}\le M
\quad\text{for all }\alpha\in(0,\alpha_0).
\end{equation}
From the compactness of $\mathcal{M}_\alpha(\Omega)\cap {\mathcal{S}}_1$ (Corollary~\ref{co:compactness}), the constant $M$ can be taken independent of the extremal $u_*$.

Next, we compare the Rayleigh quotients of the three eigenvalue problems \eqref{eq:dirichlet_eigenvalue}, \eqref{eq:alpha_eigenvalue}, and \eqref{eq:TM_eigenvalue}.
For any $\phi\in W^{1,2}_0(\Omega)\setminus\{0\}$, using \eqref{eq:uinftybound} we have
\[
\frac{1}{(1 + 2\alpha M^2) e^{\alpha M^2}}
\frac{\int_{\Omega} |\nabla \phi|^2 \ud x}{\int_{\Omega}  \phi^2 \ud x}
\le
\frac{\int_{\Omega} |\nabla \phi|^2 \ud x}{\int_{\Omega} (1 + 2\alpha u_*^2) e^{\alpha u_*^2} \phi^2 \ud x}
\le
\frac{\int_{\Omega} |\nabla \phi|^2 \ud x}{\int_{\Omega} e^{\alpha u_*^2} \phi^2 \ud x}
\le
\frac{\int_{\Omega} |\nabla \phi|^2 \ud x}{\int_{\Omega}  \phi^2 \ud x}.
\]
Taking the infimum over all $\phi\in W^{1,2}_0(\Omega)\setminus\{0\}$ gives us
\[
\frac{\mu_1}{(1 + 2\alpha M^2) e^{\alpha M^2}}
\ \le\
\lambda_1(\alpha)
\ \le\
\Lambda_1(\alpha)
\ \le\
\mu_1.
\]
Since $(1 + 2\alpha M^2) e^{\alpha M^2}\to1$ as $\alpha\to0$, it follows that
\[
\lim_{\alpha\to0}\lambda_1(\alpha)
=\lim_{\alpha\to0}\Lambda_1(\alpha)
=\mu_1.
\]
Thus, for every $\delta>0$, there exists $\alpha_0>0$ such that $\alpha\in(0,\alpha_0)$ implies
\[
\mu_1-\delta<\lambda_1(\alpha)<\mu_1,
\]
which completes the proof.
\end{proof}

\begin{lemma}[Spectral gap]\label{lm:spectral_gap}
    Let $\Omega\subset\mathbf{R}^2$ be a smooth, bounded domain.
    For $\alpha\in(0,\alpha_0)$ sufficiently small, one has
    \[
    \lambda_1(\alpha) < \Lambda_1(\alpha) < \lambda_2(\alpha).
    \]
    In particular, $\Lambda_1(\alpha)$ is not an eigenvalue of the linearized problem \eqref{eq:alpha_BVP}.
\end{lemma}

\begin{proof}
The strict inequality $\lambda_1(\alpha)<\Lambda_1(\alpha)$ is established in Lemma~\ref{lm:eigenvalue_convergence}.
For the upper bound, we extend the Rayleigh quotient comparison to higher eigenvalues.
By the Courant--Fischer min-max principle (see, e.g., \cite[Chapter~VI]{CourantHilbert}), the $j$-th eigenvalue of \eqref{eq:alpha_BVP} satisfies
\[
\lambda_j(\alpha)
=\inf_{\substack{V\subset W^{1,2}_0(\Omega)\\ \dim V=j}}\;
\sup_{v\in V\setminus\{0\}}
\frac{\int_\Omega|\nabla v|^2\ud x}
{\int_\Omega(1+2\alpha u_*^2)e^{\alpha u_*^2}v^2\ud x}.
\]
Using the uniform bound $\|u_*\|_{L^\infty(\Omega)}\le M$ from \eqref{eq:uinftybound}, the same sandwich estimate as in Lemma~\ref{lm:eigenvalue_convergence} yields
\[
\frac{\mu_j}{(1+2\alpha M^2)e^{\alpha M^2}}
\le \lambda_j(\alpha) \le \mu_j
\quad\text{for all }j\in\mathbf{N}.
\]
In particular, $\lambda_j(\alpha)\to\mu_j$ as $\alpha\to0$ for every $j\ge1$.
Since $\mu_1<\mu_2$ (the first Dirichlet eigenvalue is simple) and $\Lambda_1(\alpha)\to\mu_1$, there exists $\alpha_0>0$ such that
\[
\Lambda_1(\alpha)<\tfrac{1}{2}(\mu_1+\mu_2)<\lambda_2(\alpha)
\quad\text{for all }\alpha\in(0,\alpha_0).
\]
Combined with Lemma~\ref{lm:eigenvalue_convergence}, this gives the desired spectral gap.
Since the eigenvalues $\{\lambda_j(\alpha)\}_{j\ge1}$ are discrete and $\Lambda_1(\alpha)$ lies strictly between consecutive eigenvalues, it cannot belong to the spectrum of \eqref{eq:alpha_BVP}.
\end{proof}

\subsection{Taylor expansion for almost optimizers}\label{subsec:taylor_expansion}
We obtain an asymptotic expansion for the Trudinger--Moser functional near optimizers.

\begin{lemma}[Expansion near optimizers]\label{lm:taylor_expansion}
Let $\Omega\subset\mathbf{R}^2$ be a smooth, bounded domain.
For any $u_*\in \mathcal{M}_\alpha(\Omega)\cap \mathcal{S}_1$ and $u\in \mathcal{S}_1$ one has 
\begin{equation} \label{eq:taylor_expansion} 
\mathscr{E}_{\alpha,\Omega} (u_*) - \mathscr{E}_{\alpha,\Omega} (u) 
= \frac{\alpha}{\Lambda_1(\alpha)} 
		\| \nabla w\|_{L^2(\Omega)}^2 - \alpha\mathcal{I}^*_{\alpha,\Omega}(w) 
        + \mathcal{O}(\|\nabla w\|_{L^2(\Omega)}^3) \quad {\rm as} \quad \|\nabla w\|_{L^{2}(\Omega)}\to0,
\end{equation} 
where $w=u-u_*$ and
\begin{equation}\label{eq:integral_error}
  \mathcal{I}^*_{\alpha,\Omega}(w)= \int_\Omega w^2(1+2\alpha u_*^2)
		e^{\alpha u_*^2} \ud x.  
\end{equation}
\end{lemma} 

\begin{proof}
Initially, a direct computation implies that the following expansion holds: 
\begin{align} \label{expansion1} 
\nonumber
\mathscr{E}_{\alpha,\Omega} (u_* + w)
&= \int_\Omega e^{\alpha u_*^2} \sum_{\ell =0}^\infty \frac{\alpha^\ell}{\ell !}
(2u_* w + w^2)^\ell \ud x \\ \nonumber 
& = \int_\Omega e^{\alpha u_*^2} \ud x + \alpha \int_\Omega e^{\alpha u_*^2} 
(2u_* w + w^2) \ud x + \frac{\alpha^2}{2} \int_\Omega e^{\alpha u_*^2} (4u_*^2 w^2 
+ 4u_* w^3 + w^4) \ud x\\\nonumber
&+ \sum_{\ell=3}^\infty \frac{\alpha^\ell}{\ell !} \int_\Omega e^{\alpha u_*^2} 
(2u_* w + w^2)^\ell \ud x\\
&:=\mathscr{E}_{\alpha,\Omega} (u_*)+\alpha\mathcal{I}^*_{\alpha,\Omega}(w)+T_1(w)+T_2(w)+T_3(w),
\end{align} 
where $\mathcal{I}^*_{\alpha,\Omega}$ is defined in \eqref{eq:integral_error} and
\[
T_1(w) = 2\alpha \int_\Omega u_* w e^{\alpha u_*^2} \ud x,
\]
\[
T_2(w) = \int_{\Omega} \left (2\alpha^2 u_* w^3 + 
\frac{\alpha^2}{2} w^4 \right ) e^{\alpha u_*^2} \ud x, 
\]
and
\[
T_3(w) = \sum_{\ell=3}^\infty \frac{\alpha^\ell}{\ell !} \int_\Omega e^{\alpha u_*^2} 
(2u_* w + w^2)^\ell \ud x. 
\]
Then
\[\mathscr{E}_{\alpha,\Omega} (u_*) - \mathscr{E}_{\alpha,\Omega} (u) =-T_1(w)-\alpha\mathcal{I}^*_{\alpha,\Omega}(w)-T_2(w)-T_3(w).\]
We analyze each of these terms individually.
We start with the first term, for which we use the first eigenvalue of the boundary problem \eqref{eq:TM_BVP}. 

\noindent{\bf Claim 1:} $T_1(w)=-\frac{\alpha}{\Lambda_1(\alpha)}\|\nabla w\|_{L^2(\Omega)}^2$.

\noindent First, observe
that since $u_*,u\in \mathcal{S}_1$, one has
\[1 = \int_\Omega |\nabla u|^2 \ud x = \int_\Omega |\nabla u_* + \nabla w|^2 \ud x =
\int_\Omega |\nabla u_*|^2 \ud x + 2 \int_\Omega \langle \nabla u_*, \nabla w \rangle
\ud x + \int_\Omega |\nabla w|^2 \ud x,\]
which, by rearranging, implies  
\begin{equation} \label{first_order1}  
\int_\Omega |\nabla w|^2 \ud x = -2 \int_\Omega \langle \nabla u_*, \nabla w
\rangle \ud x.
\end{equation} 
Second, since $u_*\in\mathcal{M}_\alpha(\Omega)$ is an extremal of the functional $\mathscr{E}_{\alpha,\Omega}$, it solves the 
boundary value problem \eqref{eq:TM_BVP} with Lagrange multiplier $\Lambda_1(\alpha)>0$ defined in \eqref{eq:first_TM_eigenvalue}, namely 
\begin{equation}\label{euler_lag_eqn}
		 \begin{cases}
			-\Delta u_*  = \Lambda_1(\alpha)  e^{\alpha u_{*}^2} u_* & \mbox{in} \quad \Omega\\
			u_*=0 & \mbox{on} \quad\partial \Omega.
		\end{cases} 
	\end{equation}
As a consequence of \eqref{first_order1} and 
\eqref{euler_lag_eqn}, we get
\begin{eqnarray} \label{first_order2} 
T_1(w) =  \frac{2\alpha}{\Lambda_1(\alpha)} \int_\Omega \langle \nabla u_*, \nabla w \rangle \ud x = -\frac{\alpha}{\Lambda_1(\alpha)} \int_\Omega |\nabla w|^2\ud x = -\frac{\alpha}{\Lambda_1(\alpha)} 
\| \nabla w \|_{L^2(\Omega)}^2. 
\end{eqnarray} 
This proves the first estimate.

Before analyzing the terms $T_2(w)$ and $T_3(w)$,
let us perform some preliminary computations. 
First, observe that the convexity of the function $t \mapsto t^\ell$ implies that 
for each $a,b >0$ we have 
\begin{equation}\label{convex_power}
\left ( \frac{a + b}{2} \right )^\ell \leq \frac{a^\ell + b^\ell}{2} \quad {\rm or} \quad
(a+b)^\ell \leq 2^{\ell-1} (a^\ell + b^\ell ).
\end{equation} 
Next, we use the Trudinger--Moser and H\"older 
inequalities to bound integrals of powers of $w$. 
More precisely, we have 
    \begin{align*} 
    {\rm TM}(\alpha,\Omega) \geq \int_\Omega \exp \left (\frac{\alpha w^2}{\| \nabla w\|_{L^2(\Omega)}^2} \right ) \ud x 
    & = \int_\Omega \sum_{j = 0}^\infty \frac{\alpha^j}{j ! 
    \| \nabla w \|_{L^2(\Omega)}^{2j} }w^{2j}\ud x\\
    &= \sum_{j = 0}^\infty \frac{\alpha^j}
    {j ! \| \nabla w\|_{L^2(\Omega)}^{2j}} 
    \int_\Omega w^{2j} \ud x \\ 
    & \geq \frac{\alpha^\ell}{\ell! \| \nabla w \|_{L^2(\Omega)} ^{2\ell}} 
    \int_\Omega w^{2\ell} \ud x\\
    &= \frac{\alpha^{\ell} \| w\|_{L^{2\ell}(\Omega)}^{2\ell}}
    {\ell! \| \nabla w\|_{L^2(\Omega)}^{2\ell} }
    \end{align*} 
    for any specific $\ell \in \mathbf{N}$. This is justified by the fact that $w^2 \geq 0$, 
    so each term in the sum is non-negative. We can rearrange this inequality to read 
    \begin{equation} \label{higher_order1}
    \int_\Omega |w|^{2\ell}\ud x \leq {\rm TM}(\alpha,\Omega)\frac{\ell!}{\alpha^\ell}
    \| \nabla w \|_{L^2(\Omega)}^{2\ell} . \end{equation}
    We can similarly control the integral of $w^{2\ell-1}$ using H\"older's inequality as follows: 
  \begin{align} \label{higher_order2} 
    \| w \|_{L^{2\ell-1}(\Omega)}^{2\ell -1} & =  \| w^{2\ell -1} \|_{L^1(\Omega)}    \leq \| \chi_{\Omega} \|_{L^{2\ell}(\Omega)} 
    \| w^{2\ell -1} \|_{L^{\frac{2\ell}{2\ell-1}}(\Omega)} \\ \nonumber 
    & = |\Omega|^{\frac{1}{2\ell}} \left ( \int_\Omega w^{2\ell} \ud x\right )^{\frac{2\ell-1}{2\ell}} 
    = |\Omega|^{\frac{1}{2\ell}} \left ( 
    \| w \|_{L^{2\ell}(\Omega)}^{2\ell} \right )^{\frac{2\ell-1}{2\ell}} \\ \nonumber 
    & \leq {\rm TM}(\alpha,\Omega)^{\frac{2\ell-1}{2\ell}} \frac{(\ell !)^{\frac{2\ell-1}{2\ell}}}
    {\alpha^{\frac{2\ell-1}{2}}} \| \nabla w\|_{L^2(\Omega)}^{2\ell -1}.
    \end{align} 

    We can proceed to estimate the second term. 
    
    \noindent{\bf Claim 2:} $T_2(w)=\mathcal{O}(\|\nabla w\|^3_{L^2(\Omega)})$ as $\|\nabla w\|_{L^2(\Omega)}\to0$.

    \noindent
    As a matter of fact, using \eqref{higher_order1} we see 
    \begin{equation} \label{higher_order8} 
    \left | \frac{\alpha^2}{2} \int_\Omega e^{\alpha u_*^2} w^4 \ud x \right | \leq 
    {\rm TM}(\alpha,\Omega) |\Omega |e^{\alpha M^2} \| \nabla w \|_{L^2(\Omega)}^4 \lesssim_{\alpha, \Omega, M}\| \nabla w\|_{L^2(\Omega)}^4.\end{equation} 
    Additionally, by means of \eqref{higher_order2}, we get
    \begin{equation} \label{higher_order9} 
    \left | 2\alpha^2 \int_\Omega u_*w^3 e^{\alpha u_*^2} \ud x \right | \leq 
    2^{7/4} \alpha ^{1/4} {\rm TM}(\alpha,\Omega)^{3/4} M e^{\alpha M^2} 
    \| \nabla w\|_{L^2(\Omega)}^3\lesssim_{\alpha, \Omega, M}\| \nabla w\|_{L^2(\Omega)}^3.\end{equation}
    This concludes the proof of the claim.

    At last, we can estimate the third term, which is slightly more technical.
    
    \noindent{\bf Claim 3:} $T_3(w)=\mathcal{O}(\|\nabla w\|^3_{L^2(\Omega)})$ as $\|\nabla w\|_{L^2(\Omega)}\to0$.

    \noindent Now, let $M = \max \{ 1, \| u_*\|_{L^\infty (\Omega)} \}$. Using \eqref{convex_power}
    we can bound $T_3(w)$ by the following three series: 
    \begin{align} \label{higher_order3} 
    |T_3(w)| & =  \left | \sum_{\ell=3}^\infty \frac{\alpha^\ell}{\ell !} \int_\Omega 
    e^{\alpha u_*^2} (2u_* w + w^2)^{\ell} \ud x \right | \\ \nonumber 
    & \leq  e^{\alpha M^2} \sum_{\ell=3}^\infty  \frac{\alpha^\ell 2^{\ell-1}}{\ell !} 
    \int_\Omega (2^\ell u_*^\ell |w|^\ell + |w|^{2\ell} ) \ud x \\ \nonumber 
    & \leq e^{\alpha M^2} \left [ \sum_{\ell=2}^\infty \frac{\alpha^{2\ell-1} 2^{4\ell-3} M^{2\ell-1} }
    {(2\ell-1) !} \int_\Omega |w|^{2\ell-1} \ud x + \sum_{\ell=2}^\infty \frac{\alpha^{2\ell} 2^{4\ell-1}
    M^{2\ell} }{(2\ell)!} \int_\Omega |w|^{2\ell} \ud x  \right ] \\ \nonumber 
    &+ e^{\alpha M^2} \sum_{\ell=3}^\infty \frac{\alpha^\ell 2^{\ell-1}}{\ell !}
    \int_\Omega |w|^{2\ell} \ud x\\ \nonumber 
    &:= e^{\alpha M^2} [S_1(w)+ S_2(w)+S_3(w)]. 
    \end{align} 
    Here we set
    \[
    S_1(w):= \sum_{\ell=2}^\infty \frac{\alpha^{2\ell-1} 2^{4\ell-3} M^{2\ell-1} }
    {(2\ell-1) !} \int_\Omega |w|^{2\ell-1} \ud x,
    \]
    \[
    S_2(w):=\sum_{\ell = 2}^\infty \frac{\alpha^{2\ell} 2^{4\ell-1}M^{2\ell} }{(2\ell)!}
    \int_\Omega |w|^{2\ell} \ud x,
    \]
    and
    \[
    S_3(w):=\sum_{\ell=3}^\infty \frac{\alpha^\ell 2^{\ell-1}}{\ell !}
    \int_\Omega |w|^{2\ell} \ud x.
    \]
    \noindent Next, we estimate each sum individually. First, we use \eqref{higher_order1} to bound the 
    last sum by a geometric series, so long as 
    $\| \nabla w \|_{L^2(\Omega)}^2 < \frac{1}{2} .$
    More precisely, one has
    \begin{align*} 
    S_3(w) & = \sum_{\ell=3}^\infty \frac{\alpha^\ell 2^{\ell-1}}{\ell !}
    \int_\Omega |w|^{2\ell} \ud x \\ \nonumber 
    & \leq {\rm TM}(\alpha,\Omega) \frac{|\Omega |}{2} \sum_{\ell=3}^\infty(2\| 
    \nabla w\|_{L^2(\Omega)}^2)^\ell \\ \nonumber 
    & = {\rm TM}(\alpha,\Omega)\frac{|\Omega|}{2} \left ( \frac{1}{1-2\|\nabla w\|_{L^2(\Omega)}^2}
    - 1 - 2\| \nabla w\|_{L^2(\Omega)}^2 - 4\| \nabla w\|_{L^2(\Omega)}^4 \right ) \\ \nonumber 
    & = {\rm TM}(\alpha,\Omega)\left(\frac{4\| \nabla w \|_{L^2(\Omega)}^6}{1-2\| \nabla w \|_{L^2(\Omega)}^2}\right). 
    \end{align*} 
    From this, we conclude
    \begin{equation} \label{higher_order4}
    |S_3(w)| \lesssim_{\alpha,\Omega} \| \nabla w \|_{L^2(\Omega)}^4.
    \end{equation}
    We can also use \eqref{higher_order1} to bound the second sum when $0<\alpha\ll1$. Also, when $\| \nabla w \|_{L^2(\Omega)} < 1$, we have  
    \begin{eqnarray*} 
    S_2(w) & = & \sum_{\ell = 2}^\infty \frac{\alpha^{2\ell} 2^{4\ell-1}M^{2\ell} }{(2\ell)!} 
    \int_\Omega |w|^{2\ell} \ud x \\  
    & \leq & {\rm TM}(\alpha,\Omega) |\Omega| \sum_{\ell=3}^\infty \frac{\alpha^\ell 2^{4\ell-1} \ell! M^{2\ell}}
    {(2\ell)!} \| \nabla w\|_{L^2(\Omega)}^{2\ell} \\  
    & \leq & {\rm TM}(\alpha,\Omega) \| \nabla w\|_{L^2(\Omega)}^4 \sum_{\ell=2}^\infty 
    \frac{\alpha^\ell 2^{4\ell-1} \ell! M^{2\ell}}{(2\ell)!}. 
    \end{eqnarray*} 
    One can show, using the ratio test, that this last sum converges; and so
    \begin{equation} \label{higher_order5} 
    |S_2(w)| \lesssim_{\alpha,\Omega,M} \| \nabla w \|_{L^2(\Omega)}^4.
    \end{equation}
    At last, we use \eqref{higher_order2} to bound $S_1(w)$. 
    In this case, we have 
    \begin{align*} 
    S_1(w) & = \sum_{\ell=2}^\infty \frac{\alpha^{2\ell-1} 2^{4\ell-3} M^{2\ell-1} }
    {(2\ell-1) !} \int_\Omega |w|^{2\ell-1} \ud x \\ 
    & \leq |\Omega| \sum_{\ell=2}^\infty \frac{\alpha^{\frac{2\ell-2}{2}} 2^{4\ell-3} 
    M^{2\ell-1} {\rm TM}(\alpha,\Omega)^{\frac{2\ell-1}{2\ell}} (\ell!)^{\frac{2\ell-1}{2\ell}}}
    {(2\ell-1)!} \| \nabla w\|_{L^2(\Omega)}^{2\ell-1} \\ 
    & \leq  \| \nabla w \|_{L^2(\Omega)}^3 \sum_{\ell=2}^\infty 
    {\rm TM}(\alpha,\Omega)^{\frac{2\ell-1}{2\ell}}\frac{\alpha^{\frac{2\ell-2}{2}} 2^{4\ell-3} M^{2\ell-1} (\ell!)^{\frac{2\ell-1}{2\ell}}}{(2\ell-1)!}. 
    \end{align*} 
    Once again, one can use the ratio test to show that this last sum converges, which in 
    turn gives us 
    \begin{equation} \label{higher_order6} 
    S_1(w) \lesssim_{\alpha,\Omega,M} \| \nabla w \|_{L^2(\Omega)}^3. 
    \end{equation} 
    Combining \eqref{higher_order4}, \eqref{higher_order5} and \eqref{higher_order6}
    we now see 
    \begin{equation} \label{higher_order7} 
    |T_3(w)| \lesssim_{\alpha, \Omega, M} \| \nabla w \|_{L^2(\Omega)}^3.
    \end{equation} 
    This concludes the proof of the claim.
    
    Finally, by combining \eqref{first_order2}, \eqref{higher_order7}, \eqref{higher_order8}, 
    and \eqref{higher_order9} from Claims 1, 2, 3, respectively, we obtain 
    the desired asymptotic expansion \eqref{eq:taylor_expansion}. 
    The lemma is proved.
\end{proof}

\subsection{Second variation}\label{subsec:second_variation}

Now, we prove some properties of the second variation of the Trudinger--Moser functional.

Let $u_*\in\mathcal{M}_\alpha(\Omega)\cap\mathcal{S}_1$ be a fixed optimizer and
$\{u_k\}\subset\mathcal{S}_1$ be a violating sequence with $u_k\to u_*$ in
$W^{1,2}_0(\Omega)$, that is, \eqref{contra0} holds. Let $\{w_k\}_{k\in\mathbf{N}}$ be the sequence of deviations $w_k=u_k-u_*$, and denote by $\{\widehat{w}_k\}_{k\in\mathbf{N}}$ its normalization defined in \eqref{eq:normalized_deviations}.

We introduce the Lagrangian functional
\begin{equation}\label{eq:Lagrangian}
\mathscr{F}_{\alpha,\Omega}(u)
:=\frac{1}{2}\int_{\Omega}|\nabla u|^2\ud x
-\frac{\Lambda_1(\alpha)}{2\alpha}\int_{\Omega}e^{\alpha u^2}\ud x,
\end{equation}
whose critical points on $\mathcal{S}_1$ correspond to solutions of the
Euler--Lagrange equation~\eqref{eq:TM_BVP}.
The tangent space to $\mathcal{S}_1$ at $u_*$ is
\[
T_{u_*}\mathcal{S}_1
:=\Bigl\{\,w\in W^{1,2}_0(\Omega):
\int_\Omega\langle \nabla u_*,\nabla w\rangle \ud x=0\,\Bigr\}.
\]

The first variation of $\mathscr{F}_{\alpha,\Omega}$ at $u_*$ is
\[
\mathscr{F}_{\alpha,\Omega}'(u_*)(v)
=\int_{\Omega}\langle \nabla u_*,\nabla v\rangle \ud x
-\Lambda_1(\alpha)\int_{\Omega}e^{\alpha u_*^2}u_*v\ud x,
\]
and the second variation is the symmetric bilinear form
\[
\mathscr{F}_{\alpha,\Omega}''(u_*)(w,v)
=\int_{\Omega}\langle \nabla w,\nabla v\rangle \ud x
-\Lambda_1(\alpha)\int_{\Omega}(1+2\alpha u_*^2)
e^{\alpha u_*^2}wv\ud x.
\]
We denote this quadratic form by
\begin{equation}\label{eq:weighted-ip}
\mathscr{Q}_{\alpha,\Omega}(w,v)
:=\int_\Omega\langle \nabla w,\nabla v\rangle \ud x
-\Lambda_1(\alpha)\,\langle w,v\rangle_{\alpha,\Omega},
\quad {\rm with} \quad
\langle w,v\rangle_{\alpha,\Omega}
:=\int_\Omega(1+2\alpha u_*^2)e^{\alpha u_*^2}wv\ud x.
\end{equation}

With these notations, we can summarize the structural properties of the second variation.

\begin{lemma}[Second variation kernel]\label{lem:Qalpha-kernel-extension}
Let $\Omega\subset\mathbf{R}^2$ be a smooth bounded domain and $\alpha\in(0,4\pi]$.
Assume that $v_*\in\mathcal{S}_1$ is the first (normalized) eigenfunction of
\eqref{eq:alpha_BVP} and that $u_*\in\mathcal{M}_\alpha(\Omega)\cap\mathcal{S}_1$
is a Trudinger--Moser optimizer.
If $\{u_k\}\subset\mathcal{S}_1$ is an optimizing sequence such that
$u_k\to u_*$ in $W^{1,2}_0(\Omega)$, then the following properties hold:
\begin{enumerate}[label=\textup{(\roman*)}]
\item For all $w\in T_{u_*}\mathcal{S}_1$, one has $\mathscr{Q}_{\alpha,\Omega}(w,w)\ge0$
\item For each $\widehat w_*
\in T_{u_*}\mathcal{S}_1$ satisfying
\(\mathscr{Q}_{\alpha,\Omega}(\widehat w_*,\widehat w_*)=0\), it follows that $\mathscr{Q}_{\alpha,\Omega}(\widehat w_*,v)=0$ for any $v\in T_{u_*}\mathcal{S}_1$, that is,
\begin{equation}\label{eq:EL-on-tangent}
\int_\Omega\langle \nabla\widehat w_*,\nabla v\rangle \ud x
=\Lambda_1(\alpha)\int_\Omega(1+2\alpha u_*^2)e^{\alpha u_*^2}
\widehat w_*v\ud x
\quad {\rm for \ all} \quad v\in T_{u_*}\mathcal{S}_1.
\end{equation}
\item One has $v_*\notin T_{u_*}\mathcal{S}_1$ and $W^{1,2}_0(\Omega)
=T_{u_*}\mathcal{S}_1\oplus\mathrm{span}\{v_*\}.$
\end{enumerate}
\end{lemma}

\begin{proof}
We divide the proof into several steps as follows.

\noindent\textrm{(i)}
Since $u_*$ is a global minimizer of $\mathscr{F}_{\alpha,\Omega}$ on the
Hilbert manifold $\mathcal{S}_1$, its second variation is nonnegative on
$T_{u_*}\mathcal{S}_1$.
Thus $\mathscr{Q}_{\alpha,\Omega}(w,w)\ge0$ for all
$w\in T_{u_*}\mathcal{S}_1$.

\noindent\textrm{(ii)}
From (i), $\widehat w_*$ is a minimizer of
$\mathscr{Q}_{\alpha,\Omega}(\cdot,\cdot)$ on $T_{u_*}\mathcal{S}_1$, and
so it is also a critical point, that is,
$\mathscr{Q}_{\alpha,\Omega}(\widehat w_*,v)=0$ for any $v\in T_{u_*}\mathcal{S}_1$.

\noindent\textrm{(iii)}
For the positive first Krein--Rutman eigenfunction $v_*>0$ of~\eqref{eq:alpha_BVP}
(normalized by $\|\nabla v_*\|_{L^2(\Omega)}=1$), since $v_*$ solves \eqref{eq:alpha_BVP} we have $\|\nabla v_*\|_{L^2}^2=\lambda_1(\alpha)\langle v_*,v_*\rangle_{\alpha,\Omega}$, and thus a direct computation gives
\[
\mathscr{Q}_{\alpha,\Omega}(v_*,v_*)
=\|\nabla v_*\|_{L^2}^2 - \Lambda_1(\alpha)\langle v_*,v_*\rangle_{\alpha,\Omega}
=\left(1-\frac{\Lambda_1(\alpha)}{\lambda_1(\alpha)}\right)<0,
\]
where we used $\lambda_1(\alpha)<\Lambda_1(\alpha)$ (cf.\ Lemma~\ref{lm:spectral_gap}).
By (i), $\mathscr{Q}_{\alpha,\Omega}(w,w)\ge0$ for all $w\in T_{u_*}\mathcal{S}_1$,
hence $v_*\notin T_{u_*}\mathcal{S}_1$.
Since $T_{u_*}\mathcal{S}_1$ is a closed codimension-one subspace of
$W^{1,2}_0(\Omega)$, the direct sum decomposition in (iii) follows.
\end{proof}

\begin{lemma}[Coercivity on the tangent space]\label{lem:Q_coercivity}
Let $\Omega\subset\mathbf{R}^2$ be a smooth, bounded domain.
Let $u_*\in\mathcal{M}_\alpha(\Omega)\cap\mathcal{S}_1$ be a Trudinger--Moser optimizer and
let $v_*\in\mathcal{S}_1$ denote the first (positive, normalized) eigenfunction
of the linearized problem~\eqref{eq:alpha_BVP}.
Then, for $\alpha\in(0,\alpha_0)$ with $\alpha_0>0$ given by Lemma~\ref{lm:spectral_gap},
there exists a constant $c(\alpha)>0$ such that
\begin{equation}\label{eq:Q_coercivity}
\mathscr{Q}_{\alpha,\Omega}(w,w)\ge c(\alpha)\,\|w\|_{L^2(\Omega)}^2
\quad\text{for all }w\in T_{u_*}\mathcal{S}_1.
\end{equation}
In particular, the quadratic form $\mathscr{Q}_{\alpha,\Omega}$ has trivial kernel on $T_{u_*}\mathcal{S}_1$.
\end{lemma}

\begin{proof}
By Lemma~\ref{lem:Qalpha-kernel-extension}\textup{(iii)},
$v_*\notin T_{u_*}\mathcal{S}_1$ and
$W^{1,2}_0(\Omega)=T_{u_*}\mathcal{S}_1\oplus\mathrm{span}\{v_*\}$.
We decompose $w=t\,v_*+w^\perp$ with
$\langle w^\perp,v_*\rangle_{\alpha,\Omega}=0$.
Since $v_*$ solves~\eqref{eq:alpha_BVP} with eigenvalue $\lambda_1(\alpha)$
and the cross terms vanish by orthogonality, a direct computation gives us
\begin{align*}
\mathscr{Q}_{\alpha,\Omega}(w,w)
&=t^2\left(\lambda_1(\alpha)-\Lambda_1(\alpha)\right)
\langle v_*,v_*\rangle_{\alpha,\Omega}
+\int_\Omega|\nabla w^\perp|^2\,\ud x
-\Lambda_1(\alpha)\langle w^\perp,w^\perp\rangle_{\alpha,\Omega}.
\end{align*}
By the Courant--Fischer characterization of $\lambda_2(\alpha)$, one has
\[
\int_\Omega|\nabla w^\perp|^2\,\ud x
\ge\lambda_2(\alpha)\,\langle w^\perp,w^\perp\rangle_{\alpha,\Omega},
\]
which yields
\begin{equation}\label{eq:Q_lower}
\mathscr{Q}_{\alpha,\Omega}(w,w)
\ge t^2\left(\lambda_1(\alpha)-\Lambda_1(\alpha)\right)
\langle v_*,v_*\rangle_{\alpha,\Omega}
+\left(\lambda_2(\alpha)-\Lambda_1(\alpha)\right)
\langle w^\perp,w^\perp\rangle_{\alpha,\Omega}.
\end{equation}
We now use the tangent constraint to control the coefficient $t$.
Since $w\in T_{u_*}\mathcal{S}_1$, one has
$\int_\Omega\nabla u_*\!\cdot\!\nabla w\,\ud x=0$, which by the
Euler--Lagrange equation~\eqref{eq:TM_BVP} reads
$\int_\Omega e^{\alpha u_*^2}u_*\,w\,\ud x=0$.
We set $f:=u_*/(1+2\alpha u_*^2)$, so that the tangent constraint becomes
$\langle w,f\rangle_{\alpha,\Omega}=0$, from which we obtain
\[
t=-\frac{\langle w^\perp,f\rangle_{\alpha,\Omega}}
{\langle v_*,f\rangle_{\alpha,\Omega}}.
\]
We observe that $\langle v_*,f\rangle_{\alpha,\Omega}
=\int_\Omega v_*\,u_*\,e^{\alpha u_*^2}\,\ud x>0$
since $v_*>0$ and $u_*>0$ in $\Omega$.
It follows from the Cauchy--Schwarz inequality that
\[
t^2\le\frac{\langle f,f\rangle_{\alpha,\Omega}}
{\langle v_*,f\rangle_{\alpha,\Omega}^2}\,
\langle w^\perp,w^\perp\rangle_{\alpha,\Omega}
=:C_0(\alpha)\,\langle w^\perp,w^\perp\rangle_{\alpha,\Omega}.
\]
Substituting into~\eqref{eq:Q_lower}, one has
\begin{align*}
\mathscr{Q}_{\alpha,\Omega}(w,w)
&\ge\left[\lambda_2(\alpha)-\Lambda_1(\alpha)
+C_0(\alpha)\left(\lambda_1(\alpha)-\Lambda_1(\alpha)\right)
\langle v_*,v_*\rangle_{\alpha,\Omega}\right]
\langle w^\perp,w^\perp\rangle_{\alpha,\Omega}\\
&:=C_1(\alpha)\langle w^\perp,w^\perp\rangle_{\alpha,\Omega}.
\end{align*}
By Lemma~\ref{lm:spectral_gap} and Lemma~\ref{lm:eigenvalue_convergence},
as $\alpha\to0$ one has $\lambda_j(\alpha)\to\mu_j$ for $j=1,2$,
$\Lambda_1(\alpha)\to\mu_1$, and $C_0(\alpha)\to C_0(0)<\infty$.
Therefore, one has the convergence
\[
\lim_{\alpha\to0}C_1(\alpha)=\mu_2-\mu_1>0,
\]
which is the Dirichlet spectral gap of $\Omega$,
and in particular remains strictly positive for $\alpha\in(0,\alpha_0)$ with $0<\alpha_0\ll1$
sufficiently small.
Since $t^2\le C_0(\alpha)\,\langle w^\perp,w^\perp\rangle_{\alpha,\Omega}$,
it follows that
\[
\|w\|_{L^2(\Omega)}^2\lesssim\langle w^\perp,w^\perp\rangle_{\alpha,\Omega},
\]
and estimate~\eqref{eq:Q_coercivity} is proved.
\end{proof}

\subsection{General subcritical case}\label{subsec:generic_deg}
    At last, we can present the main result of this section, which is the proof of the stability estimate for small subcritical growth parameters. 

    \begin{proof}[Proof of Theorem~\ref{thm:subcritical_stability}]
        The proof is by contradiction. 
        Suppose that there exists a sequence $\{u_k\}_{k\in\mathbf{N}} \subset \mathcal{S}_1$, an optimizing sequence for the Trudinger--Moser inequality, such that $\|\nabla u_k\|_{L^2(\Omega)}= 1$ and \eqref{contra0} holds.
        Notice that by Lemma~\ref{lm:sphere_to_ball}, it is enough to work on the unit sphere $\mathcal{S}_1\subset W^{1,2}_0(\Omega)$.
    
	From Lemma~\ref{lm:compactness}, up to a subsequence there exists a maximizer $u_*\in \mathcal{M}_\alpha(\Omega)\cap \mathcal{S}_1$ such that $u_k \rightarrow u_*$ in $W^{1,2}_0(\Omega)$ and \eqref{contra0} still hold, that is,
    \[
    \lim_{k\rightarrow \infty}\frac{ \mathscr{E}_{\alpha,\Omega}(u_*) - \mathscr{E}_{\alpha,\Omega}(u_k)}{\|{\nabla (u_*-u_k)}\|_{L^2(\Omega)}^2}=0.
    \]
    In particular, since 
    \[{\rm dist}_\alpha(u_k,\mathcal{M}_\alpha(\Omega))=\inf_{v \in \mathcal{M}_\alpha(\Omega)} \|{\nabla (v-u_k)}\|_{L^2(\Omega)}^2 \leq\|{\nabla (u_*-u_k)}\|_{L^2(\Omega)}^2={\rm dist}_\alpha(u_k,u_*) \quad {\rm for} \quad k\in \mathbf{N},\]
    it follows
	\[
	\frac{ {\rm TM}(\alpha,\Omega) - \mathscr{E}_{\alpha,\Omega}(u_k)}{ \inf_{v \in \mathcal{M}_\alpha(\Omega)} \|{\nabla (v-u_k)}\|_{L^2(\Omega)}^2}\geq \frac{\mathscr{E}_{\alpha,\Omega}(u_*) - \mathscr{E}_{\alpha,\Omega}(u_k)}{\|{\nabla (u_*-u_k  )}\|_{L^2(\Omega)}^2}\rightarrow 0 \quad {\rm as} \quad k\rightarrow\infty.
	\]
	
	Again, by the maximum principle, we may assume without loss of generality that $u_*>0$. Next, we set $w_k=u_k-u_*$ and $\{\widehat{w}_k\}_{k\in\mathbf{N}}$ given by \eqref{eq:normalized_deviations}.
    Using Lemma~\ref{lm:taylor_expansion}, we can expand the Trudinger--Moser functional near an optimizer to obtain 
	\begin{equation} \label{final_estimate1}  
		\mathscr{E}_{\alpha,\Omega} (u_k) =  \mathscr{E}_{\alpha,\Omega}(u_*) - \frac{\alpha}{\Lambda_1(\alpha)}
		\| \nabla w_k\|_{L^2(\Omega)}^2 + \alpha \mathcal{I}_{\alpha,\Omega}^*(w_k) + \mathcal{O}(\|\nabla w_k\|_{L^2(\Omega)}^3) \quad {\rm as} \quad k\to \infty.
	\end{equation} 
   
    \noindent{\bf Claim 1:} There exists $\widehat{w}_*\in \overline{\mathcal{B}}_1$ such that $\widehat{w}_k\rightharpoonup \widehat{w}_*$ weakly and $\widehat{w}_*$ satisfies 
    \[
    \mathcal{I}_{\alpha,\Omega}^*(\widehat{w}_*)=\int_\Omega \widehat{w}_*^2(1+2\alpha u_*^2)e^{\alpha u_*^2} \ud x=\frac{1}{\Lambda_1(\alpha)} \quad {\rm for} \quad \alpha\in (0,\alpha_0).
    \]

    \noindent First, since $\sup_{k\in\mathbf{N}}\|\nabla w_k\|_{L^2(\Omega)}<\infty$, there exists $\widehat{w}_*\in \overline{\mathcal{B}}_1$ such that $\widehat{w}_k\rightharpoonup \widehat{w}_*\in W^{1,2}_0(\Omega)$, 
    up to a subsequence.
    Moreover, by the lower semi-continuity of the norm, one has
    \[
    0\leq\|{\nabla  \widehat{w}_*}\|_{L^2(\Omega)}\leq \liminf_{k\to\infty} \|{\nabla \widehat{w}_k}\|_{L^2(\Omega)}=1.
    \] 
    Next, by the compact Sobolev embedding $W^{1,2}_0(\Omega)\hookrightarrow L^2(\Omega)$ applied to \eqref{final_estimate1}, we get
	\begin{equation}\label{vai0} 
		\lim_{k \rightarrow \infty} \frac{\mathscr{E}_{\alpha,\Omega} (u_*) - \mathscr{E}_{\alpha,\Omega} (u_k)}
		{\| \nabla w_k \|_{L^2(\Omega)}^2}= \frac{\alpha}{\Lambda_1(\alpha)}\left(1 
		- \Lambda_1(\alpha) \int_\Omega \widehat{w}_*^2(1+2\alpha u_*^2)
		e^{\alpha u_*^2} \ud x\right) = 0,
	\end{equation}
    which, in turn, proves the claim.

    \noindent{\bf Claim 2:} The convergence $\widehat{w}_k\rightarrow \widehat{w}_*$ is strong in $W^{1,2}_0(\Omega)$, with $\widehat{w}_*\in T_{u_*}\mathcal{S}_1$, $\|\nabla \widehat{w}_*\|_{L^2(\Omega)}=1$, and $\mathscr{Q}_{\alpha,\Omega}(\widehat{w}_*,\widehat{w}_*)=0$.
	
    \noindent Indeed, a routine computation, as in \eqref{first_order1}, yields
	\[
	\int_\Omega \langle \nabla  u_*,  \nabla  \widehat{w}_*\rangle \ud x=\lim_{k\rightarrow \infty}\int_\Omega  \langle \nabla  u_* ,\nabla \widehat{w}_k\rangle \ud x= -\lim_{k\rightarrow \infty}\frac{1}{2}\|{\nabla w_k}\|_{L^2(\Omega)} = 0,
	\]
	which implies that $\widehat{w}_*\in T_{u_*}\mathcal{S}_1$. 
    Recall the Lagrangian functional $\mathscr{F}_{\alpha,\Omega}$ defined in~\eqref{eq:Lagrangian}, whose Euler--Lagrange equation is~\eqref{eq:TM_BVP}. Its Fréchet derivative
    $\mathscr{F}'_{\alpha,\Omega}(u_*):W^{1,2}_0(\Omega)\to \mathbf{R}$ is given by
    \[
    \mathscr{F}_{\alpha,\Omega}'(u_*)(v)=-\int_{\Omega}  v\Delta u_*   \ud x +\Lambda_1(\alpha) \int_{\Omega} u_*ve^{\alpha u_*^2} \ud x.
    \]
	Furthermore, one can obtain the second derivative $\mathscr{F}_{\alpha,\Omega}'':T_{u_*}\mathcal{S}_1\times T_{u_*}\mathcal{S}_1\to \mathbf{R}$ as
	\[
	\mathscr{F}_{\alpha,\Omega}''(u_*)(w,v)= \int_{\Omega}  \langle \nabla  w, \nabla  v \rangle \ud x- \Lambda_1(\alpha) \int_{\Omega} (1 + 2\alpha u_*^2) e^{\alpha u_*^2} wv \ud x.
	\]
	Thus, $u_*\in {\rm Crit}(\mathscr{F}_{\alpha,\Omega})$ is a critical point and, moreover, a global minimizer for the $\mathscr{F}_{\alpha,\Omega}$ over the Hilbert manifold $\mathcal{S}_1\subset W^{1,2}_0(\Omega)$; this implies 
	\begin{equation}\label{positiv}
		\mathscr{F}_{\alpha,\Omega}''(u_*)(w,w)= \int_{\Omega}|\nabla w|^2\ud x- \Lambda_1(\alpha) \int_{\Omega} (1 + 2\alpha u_*^2) e^{\alpha u_*^2} w^2 \ud x\geq 0 \quad {\rm for} \quad  w\in T_{u_*} \mathcal{S}_1.
	\end{equation}
    In addition, from  \eqref{vai0} and Fatou's lemma, one has
    \[
    \begin{aligned}
    0\leq \mathscr{F}_{\alpha,\Omega}''(u_*)(\widehat{w}_*,\widehat{w}_*)&= \|{\nabla  \widehat{w}_*}\|^2_{L^2(\Omega)}- \Lambda_1(\alpha) \int_\Omega \widehat{w}_*^2(1+2\alpha u_*^2)
    e^{\alpha u_*^2} \ud x\\
    &\leq 1
    - \Lambda_1(\alpha) \int_\Omega \widehat{w}_*^2(1+2\alpha u_*^2)
    e^{\alpha u_*^2} \ud x=0,
    \end{aligned}
    \]
    which implies $\|{\nabla  \widehat{w}_*}\|^2_{L^2(\Omega)}=1$, concluding the proof of the claim.

    We are now in a position to derive the desired contradiction.
    By Claim~2, the limit $\widehat{w}_*\in T_{u_*}\mathcal{S}_1$
    satisfies $\|\nabla\widehat{w}_*\|_{L^2(\Omega)}=1$ (in particular $\widehat{w}_*\not\equiv0$) and
    $\mathscr{Q}_{\alpha,\Omega}(\widehat{w}_*,\widehat{w}_*)=0$.
    On the other hand, by Lemma~\ref{lem:Q_coercivity},
    for $\alpha\in(0,\alpha_0)$ one has
    $\mathscr{Q}_{\alpha,\Omega}(w,w)\ge c(\alpha)\|w\|_{L^2(\Omega)}^2>0$
    for every $w\in T_{u_*}\mathcal{S}_1\setminus\{0\}$,
    a contradiction.
    This contradiction to assumption~\eqref{contra0} shows that
    \({\rm def}_{\alpha,\Omega}(u)\ \gtrsim_\alpha\ {\rm dist}_{\alpha,\Omega}(u)^2\)
    holds for all $u\in\mathcal{S}_1$, completing the proof of Theorem~\ref{thm:subcritical_stability}.
    \end{proof}
    
\subsection{Generic nondegenerate case}\label{subsec:generic_nondeg}
We present the proof of Theorem~\ref{thm:subcritical_stability_nondeg}, which applies only under the nondegeneracy condition.

For nondegenerate critical points, the kernel of the Hessian is entirely tangential
to the manifold of maximizers. In this situation, a Lyapunov--Schmidt reduction
(see \cite{BE,MR4427104})
shows that the Trudinger--Moser functional decreases quadratically in the normal directions,
yielding a local coercivity estimate of second order. We make this precise below.

We recall that for a maximizer $u_*\in\mathcal{M}_\alpha(\Omega)$, 
the second variation of $\mathscr{E}_{\alpha,\Omega}$ restricted to the tangent 
space $T_{u_*}\mathcal{S}_1$ is denoted by $\mathscr{Q}_{\alpha,\Omega}(u_*)[w,w]$. 
If $u_*$ is nondegenerate, there exists $\mu>0$ such that
\[
\mathscr{Q}_{\alpha,\Omega}(u_*)[w,w]\ge \mu\|\nabla w\|_{L^2(\Omega)}^2
\quad \text{for all} \quad w\in T_{u_*}\mathcal{S}_1.
\]
This property provides the spectral gap needed for local quadratic coercivity.

\begin{proposition}[Local quadratic coercivity]\label{prop:local_coercivity}
Let $\Omega\subset\mathbf{R}^2$ be a smooth, bounded domain.
If all optimizers of the Trudinger--Moser constant are nondegenerate, then there exists $\alpha_0>0$ such that for any $\alpha\in(0,\alpha_0)$, one can find constants $\varepsilon=\varepsilon(\alpha,u_*)>0$ and $c=c(\alpha,\Omega,u_*)>0$ such that if $\|\nabla(u-u_*)\|_{L^2(\Omega)}\le \varepsilon$, uniformly on $\mathcal{M}_\alpha(\Omega)$, it holds
\begin{equation}\label{eq:local_stability}
    {\rm TM}(\alpha,\Omega)-\mathscr{E}_{\alpha,\Omega}(u)
\ge
c(\alpha,\Omega,u_*)\|\nabla(u-u_*)\|_{L^2(\Omega)}^2
\quad {\rm for \ all} \quad u\in\mathcal{S}_1.
\end{equation}
\end{proposition}

\begin{proof}
By Lemma~\ref{lm:eigenvalue_convergence}, we have 
\[
\delta(\alpha):=\frac{1}{\Lambda_1(\alpha)}-\frac{1}{\lambda_1(\alpha)}>0 \quad {\rm for} \quad 0<\alpha\ll 1.
\]
Applying Lemma~\ref{lm:taylor_expansion} with $w=u-u_*$ gives us
\[
\mathscr{E}_{\alpha,\Omega}(u_*)-\mathscr{E}_{\alpha,\Omega}(u)
\geq 
\alpha\delta(\alpha)\|\nabla w\|_{L^2(\Omega)}^2 - C\|\nabla w\|_{L^2(\Omega)}^3,
\]
for some $C=C(\alpha, \Omega)>0$.
By the nondegeneracy hypothesis, it follows that $\alpha\,\delta(\alpha)\ge \mu>0$. 
Therefore, by choosing $\varepsilon>0$ such that $2\varepsilon C\le \mu$, we find
\[
\mathscr{E}_{\alpha,\Omega}(u_*)-\mathscr{E}_{\alpha,\Omega}(u)\ge \tfrac{\mu}{2}\|\nabla w\|_{L^2(\Omega)}^2 \quad {\rm for \ all} \quad u\in\mathcal{S}_1,
\]
which is the desired coercivity estimate.
This proves the desired quadratic coercivity estimate with 
$c(\alpha,\Omega,u_*)=\tfrac{\mu}{2}$ and $\varepsilon=\varepsilon(\alpha,u_*)$.

Finally, since $\mathcal{M}_\alpha(\Omega)\subset\mathcal{S}_1$ is compact and 
the map $u_*\mapsto \mathscr{Q}''_{\alpha,\Omega}(u_*)$ is continuous,
the nondegeneracy constants $\mu(u_*)>0$ admit a uniform positive lower bound on 
$\mathcal{M}_\alpha(\Omega)$. 
Furthermore, from Lemma~\ref{lm:sphere_to_ball}, one can extend the estimate for optimizers on $\overline{\mathcal{B}}_1\subset W^{1,2}_0(\Omega)$.
\end{proof}

\begin{proof}[Proof of Theorem~\ref{thm:subcritical_stability_nondeg} (stability part)]
By Lemma~\ref{lm:compactness}, the set $\mathcal{M}_\alpha(\Omega)\cap\mathcal{S}_1$ 
is compact, and so one finds finitely many balls $\mathcal{B}(u_*^j,\varepsilon_j)\subset W^{1,2}_0(\Omega)$ with $j\in\{1,\dots,N\}$ such that $\mathcal{M}_\alpha(\Omega)\subset \bigcup_{j =1}^N \mathcal{B}(u_*^j,\varepsilon_j)$ and
\[
    {\rm TM}(\alpha,\Omega)-\mathscr{E}_{\alpha,\Omega}(u)
\geq
c(\alpha,\Omega,u_*^j)\|\nabla(u-u_*^j)\|_{L^2(\Omega)}^2
\quad {\rm for \ all} \quad u\in\mathcal{B}(u_*^j,\varepsilon_j),
\]
where 
\[
\mathcal{B}(u_*^j,\varepsilon_j)=\{u\in W^{1,2}_0(\Omega) : \|\nabla(u- u_*^j)\|_{L^2(\Omega)}<\varepsilon_j\}.
\]

Let $C^*(\alpha,\Omega)=\min_{1\le j\le N} c(\alpha,\Omega,u_*^j)>0$ and $\varepsilon_*(\alpha,u_*):=\min_{1\le j\le N}\varepsilon_j(\alpha,u_*)$. 
If $u\in\mathcal{S}_1$ satisfies
${\rm dist}_{\alpha,\Omega}(u)\le \varepsilon_*$, 
then there exists $u_*^j\in\mathcal{M}_\alpha(\Omega)$ with $\|\nabla(u-u_*^j)\|_{L^2(\Omega)}\le \varepsilon_j$, 
and Proposition~\ref{prop:local_coercivity} applies. 
Otherwise, if $\mathscr{E}_{\alpha,\Omega}(u)\le{\rm TM}(\alpha,\Omega)$, then ${\rm dist}_{\alpha,\Omega}(u)$ is bounded below by a fixed positive number, yielding a trivial lower bound.
In both cases, the estimate below follows 
\[
{\rm def}_{\alpha,\Omega}(u) \ge C^*(\alpha,\Omega){\rm dist}_{\alpha,\Omega}(u)^2.
\]
Thus, the stability part of the result is proved.
\end{proof}

In the remainder of this section we justify the genericity part 
of Theorem~\ref{thm:subcritical_stability_nondeg}, showing that
for a generic choice of parameters $(\alpha,\Omega)$, all optimizers in $\mathcal{M}_\alpha(\Omega)$
are nondegenerate. 
The proof follows a standard Sard--Smale transversality scheme \cite{MR185604} 
(see also \cite{MR2560131,MR2982783}), adapted
to the present setting by pulling the problem back to a fixed reference domain to
parameterize $\Omega$ smoothly. This type of argument has now 
appeared several times in the literature, but we include it here for the 
reader's convenience.

Fix a smooth, bounded reference domain $\Omega_0\subset\mathbf{R}^2$ of unit volume.
For integers $m\ge 3$ and $\beta\in(0,1)$ let
\[
\mathcal{P}:=(0,4\pi)\times \mathcal{U}_{m,\beta} \quad {\rm with} \quad
\mathcal{U}_{m,\beta}\subset \mathcal{C}^{m,\beta}(\overline{\Omega}_0;\mathbf{R}^2)
\]
be an open $\mathcal{C}^{m,\beta}$-neighborhood of the identity consisting of $\mathcal{C}^{m,\beta}$-diffeomorphisms
$\psi:\overline{\Omega}_0\to\overline{\psi(\Omega_0)}$ with $\det (D\psi)>0$ and
$|\psi(\Omega_0)|=1$. Each $\psi\in\mathcal{U}_{m,\beta}$ determines a unit-volume domain
$\Omega_\psi:=\psi(\Omega_0)$ with smooth boundary.

For any $\psi\in\mathcal{U}_{m,\beta}$ we define the pullback isomorphism
\(\Xi_\psi: W^{1,2}_0(\Omega_\psi) \rightarrow\ W^{1,2}_0(\Omega_0)\) by
\(\Xi_\psi(u):=u\circ\psi\).
We pull back the Dirichlet energy and the
Trudinger--Moser functional to $\Omega_0$ by setting $\mathcal{A}_\psi, \mathcal{E}_{\alpha,\psi}:W^{1,2}_0(\Omega_0)\to \mathbf{R}$ as
\[
\mathcal{A}_\psi(v):=\frac12\int_{\Omega_0}\!\!\langle (D\psi)^{-T}\nabla v, (D\psi)^{-T}\nabla v\rangle \det (D\psi)\ud x,
\quad {\rm and} \quad 
\mathcal{E}_{\alpha,\psi}(v):=(\mathscr{E}_{\alpha,\Omega_\psi}\circ \Xi_\psi^{-1})(v).
\]
Both maps depend $\mathcal{C}^{m-1}$-smoothly on $\psi$ (indeed $\mathcal{C}^{m-1,\beta}$),
and are $\mathcal{C}^\infty$ in $v$ (and in $\alpha$).

We work on the fixed constraint
\[
\mathcal{S}_{1,\psi}:=\Bigl\{v\in W^{1,2}_0(\Omega_0): \mathcal{A}_\psi(v)=1\Bigr\}
\quad {\rm and} \quad
\mathcal{B}:=W^{1,2}_0(\Omega_0).
\]

Next, notice that critical points of $\mathcal{E}_{\alpha,\psi}$ on $\mathcal{S}_{1,\psi}$ are characterized by the
Euler--Lagrange equation with Lagrange multiplier $\Lambda\in\mathbf{R}$, {\it i.e.,}
\[
D_v\mathcal{E}_{\alpha,\psi}(v)-\Lambda\,D_v\mathcal{A}_\psi(v)=0 \quad  {\rm and} \quad \mathcal{A}_\psi(v)=1.
\]
Let us define the augmented map $\mathfrak{F}:\ \mathcal{P}\times \bigl(\mathcal{B}\times\mathbf{R}\bigr)\ \longrightarrow\ 
\mathcal{B}'\times\mathbf{R}$ by
\[
\mathfrak{F}(\alpha,\psi;v,\Lambda):=
\Bigl(D_v\mathcal{E}_{\alpha,\psi}(v)-\Lambda D_v\mathcal{A}_\psi(v),\; \mathcal{A}_\psi(v)-1\Bigr),
\]
where $\mathcal{B}'=W^{-1,2}(\Omega_0)$.
Then $(\alpha,\psi;v,\Lambda)$ is a zero of $\mathfrak{F}$ if and only if $v\in\mathcal{S}_{1,\psi}$ is a
critical point of $\mathcal{E}_{\alpha,\psi}$ with multiplier $\Lambda\in\mathbf{R}$.

\begin{lemma}[Regularity and Fredholm property]\label{lm:fredholm}
For $m\ge 3$ the map $\mathfrak{F}$ is $\mathcal{C}^{m-1}$ in $(\alpha,\psi;v,\Lambda)$.
For each $(\alpha,\psi;v,\Lambda)$ with $\mathfrak{F}(\alpha,\psi;v,\Lambda)=0$, the linearization 
\(D_{(v,\Lambda)}\mathfrak{F}(\alpha,\psi;v,\Lambda):\ \mathcal{B}\times\mathbf{R}\ \longrightarrow\ \mathcal{B}'\times\mathbf{R}\)
is a Fredholm operator of index $0$.
\end{lemma}

\begin{proof}
The claimed regularity follows from the $\mathcal{C}^{m-1}$-dependence of $(D\psi)^{\pm1}$ and $\det D\psi$,
and standard composition/chain rules in $\mathcal{C}^{m-1,\beta}$ and Sobolev spaces.
At a zero, the linearization is
\[
D_{(v,\Lambda)}\mathfrak{F}(\alpha,\psi;v,\Lambda)[h,\ell]
=
\Bigl(\mathscr{L}_{\alpha,\psi,v}(h)-\ell\,D_v\mathcal{A}_\psi(v),\ D_v\mathcal{A}_\psi(v)[h]\Bigr),
\]
where $\mathscr{L}_{\alpha,\psi,v}$ is the second variation (Hessian) of the Lagrangian
$(\mathcal{E}_{\alpha,\psi}-\Lambda\mathcal{A}_\psi)[v]$, acting as a symmetric elliptic operator
$\mathcal{B}\to\mathcal{B}'$. The boundary condition is Dirichlet, so $\mathscr{L}_{\alpha,\psi,v}$
is Fredholm of index $0$; the rank-one constraint row/column keeps the total index equal to $0$.
\end{proof}

We say that a critical point $v\in\mathcal{S}_{1,\psi}$ is \emph{nondegenerate} if the operator
$D_{(v,\Lambda)}\mathfrak{F}(\alpha,\psi;v,\Lambda)$ is bijective. Indeed, its kernel consists of pairs
$(h,\ell)$ such that
\[
\mathscr{L}_{\alpha,\psi,v}(h)-\ell\,D_v\mathcal{A}_\psi(v)=0
\quad {\rm and} \quad
D_v\mathcal{A}_\psi(v)[h]=0,
\]
which, after eliminating $\ell$ along the constraint tangent $T_v\mathcal{S}_{1,\psi}=
\ker D_v\mathcal{A}_\psi(v)$, reduces to $\mathscr{L}_{\alpha,\psi,v}(h)=0$ with $h\in T_v\mathcal{S}_{1,\psi}$.
Thus, bijectivity is equivalent to trivial kernel on $T_v\mathcal{S}_{1,\psi}$, {\it i.e.,} nondegeneracy.

Let us consider the projection $\pi:\mathcal{P}\times(\mathcal{B}\times\mathbf{R})\to\mathcal{P}$ and the
zero set $\mathcal{Z}:=\mathfrak{F}^{-1}(0)$. By Lemma~\ref{lm:fredholm} and the implicit function
theorem in Banach manifolds (cf. \cite[Chapter 4]{Zeidler1985}), it follows that $\mathcal{Z}$ is a $\mathcal{C}^{m-1}$-Banach submanifold and the restriction
\[
\Pi:=\pi|_{\mathcal{Z}}:\ \mathcal{Z}\ \longrightarrow\ \mathcal{P}
\]
is a $\mathcal{C}^{m-1}$ Fredholm map of index $0$. A parameter $(\alpha,\psi)\in\mathcal{P}$ is a \emph{regular
value} of $\Pi$ iff for every $(v,\Lambda)$ with $\mathfrak{F}(\alpha,\psi;v,\Lambda)=0$ the partial
derivative $D_{(v,\Lambda)}\mathfrak{F}(\alpha,\psi;v,\Lambda)$ is onto (equivalently, bijective by index
$0$). As observed above, this is precisely the nondegeneracy of all critical points of
$\mathcal{E}_{\alpha,\psi}$ on $\mathcal{S}_{1,\psi}$.

\begin{proposition}[Residual set of parameters]\label{prop:residual}
The set
\[
\mathcal{G}\ :=\ \Bigl\{(\alpha,\psi)\in\mathcal{P}:\ \text{all critical points of $\mathcal{E}_{\alpha,\psi}$
on $\mathcal{S}_{1,\psi}$ are nondegenerate}\Bigr\}
\]
is residual (countable intersection of open dense sets) in $\mathcal{P}$. In particular, it is dense.
\end{proposition}

\begin{proof}
By Sard--Smale theorem \cite[Theorem~1.3]{MR185604}, the set of regular values of the $\mathcal{C}^{m-1}$ Fredholm map $\Pi$ is residual in
$\mathcal{P}$ provided $m-1>\max\{0,\operatorname{ind}(\Pi)\}=0$. Since $m\ge 3$, this holds.
If $(\alpha,\psi)$ is a regular value, then for each zero $(v,\Lambda)$ the linearization in $(v,\Lambda)$
is onto, hence (by index $0$) invertible, {\it i.e.,} every critical point is nondegenerate.
\end{proof}

We now discuss the openness of the transversality map and the passage to geometric parameters.
In fact, residual sets are dense; openness follows from continuous dependence of the Hessian on parameters.

\begin{lemma}[Openness]\label{lem:open}
The set $\mathcal{G}\subset \mathcal{P}$ is relatively open.
\end{lemma}

\begin{proof} 
We note that nondegeneracy is equivalent to invertibility of the bounded linear operator
$D_{(v,\Lambda)}\mathfrak{F}(\alpha,\psi;v,\Lambda)$, which depends continuously on $(\alpha,\psi)$ in the operator norm.
Invertibility is an open property by the bounded inverse theorem, hence $\mathcal{G}$ is open.
\end{proof}

At last, we can conclude the proof of our second main result.
\begin{proof}[Proof of Theorem~\ref{thm:subcritical_stability_nondeg} (genericity part)]
We transfer from the diffeomorphism parameter $\psi$ to the geometric domain parameter $\Omega$.
By construction, $\psi\mapsto \Omega_\psi$ is open in the $\mathcal{C}^{m,\beta}$ topology, so the image of the
open dense set $\{(\alpha,\psi): (\alpha,\psi)\in\mathcal{G}\}$ is open and dense in the set of
$(\alpha,\Omega)$ with $\Omega=\Omega_\psi$ of class $\mathcal{C}^{m,\beta}$.

From this, we conclude that for any fixed $m\ge 3$ and among pairs $(\alpha,\Omega)$ 
with $\alpha\in(0,4\pi)$ and $\Omega$ of class $\mathcal{C}^{m,\beta}$ and unit volume, 
the subset for which all maximizers in $\mathcal{M}_\alpha(\Omega)$ are
nondegenerate is open and dense.
\end{proof}

\section{Critical regime on a disk}\label{sec:critical_regime}
In this section, we address the Trudinger--Moser inequality at the critical exponent
$\alpha=4\pi$ on the unit disk $\mathbf{D}\subset\mathbf{R}^2$.
Unlike the general subcritical case discussed in \S~\ref{subsec:generic_nondeg},
the rotational symmetry of $\mathbf{D}$ allows for a complete spectral decomposition
of the linearized operator into angular Fourier modes.
This structure permits a direct and quantitative proof of stability that bypasses
the variational contradiction strategy used earlier.
The key observation is that the first nonradial mode (corresponding to angular frequency
$k=1$) lies strictly above the radial constrained eigenvalue, providing a uniform
\emph{spectral gap} between the radial and nonradial parts of the tangent space.
This gap ensures that the second variation of the functional is strictly positive in
all nonradial directions, leading to a global quadratic coercivity estimate on
$\mathcal{S}_1$ even in the critical regime.

We recall that the Trudinger--Moser inequality in the open unit disk $\mathbf{D}\subset \mathbf{R}^2$ can be stated as
\[\sup_{\substack{\bar{u} \in W_0^{1,2}(\mathbf{D})\\ \|{\nabla \bar{u}}\|_{L^2(\mathbf{D})}=1}} {\mathscr{E}_{\alpha,\mathbf{D}}(\bar{u})}= {\rm TM}(\alpha,\mathbf{D}) \quad {\rm for} \quad \alpha\in (0,4\pi].
\]
In what follows, we write the constrained second variation as
\begin{equation}\label{eq:Qalpha-disk}
\mathscr{Q}_\alpha[w]=\int_{\mathbf{D}}|\nabla w|^2\ud x-\Lambda_1^\circ(\alpha)\int_{\mathbf{D}}(1+2\alpha u_*^2)e^{\alpha u_*^2}w^2\ud x.
\end{equation}

\subsection{Second-eigenvalue spectral gap}\label{subsec:second_eigenvalue}
Before proving the main stability theorem in the critical case, we establish a key
spectral property of the linearized operator on the disk.
It asserts that the first nonradial eigenvalue lies strictly above the constrained radial one,
providing the uniform mode gap that will yield coercivity in all nonradial directions.

Let $u_*\in\mathcal{M}_\alpha(\mathbf{D})$ be the (positive) radial optimizer, and set
\[
W_\alpha(u_*(r)):=(1+2\alpha u_*^2(r))e^{\alpha u_*^2(r)} \quad \text{for} \quad r\in(0,1).
\]
Let us decompose $w\in W^{1,2}_0(\mathbf{D})$ in Fourier modes, that is,
\[
w(r,\theta)=\sum_{k=0}^\infty\phi_k(r)\big(A_k\cos k\theta+B_k\sin k\theta\big).
\]
For each $k\in\mathbf{N}$, the radial factor $\phi_k\in \mathcal{C}^2((0,1))$ solves 
the Sturm--Liouville problem
\begin{equation}\label{eq:SLk}\tag{$SL_{\alpha, k}$}
\begin{cases}
\mathbb{L}^*_{\alpha,k} (\phi)=\lambda rW_\alpha(u_*)\phi \quad \text{in} \quad (0,1),\\
\phi(1)=\phi(0)=0,
\end{cases}
\end{equation}
where 
\[
\mathbb{L}^*_{\alpha,k} (\phi):=-(r\phi')'+V_k(r)\phi = -(r\phi')' + \frac{k^2}{r^2} \phi
\]
is the restriction of $-\Delta$ to the $k$th Fourier mode. 
We also denote by $\lambda_1^{(k)}(\alpha)$ its first eigenvalue, which, from Lemma~\ref{lem:mode_gap_disk}, satisfies
\[
\lambda_1^{(0)}(\alpha)<\Lambda_1^\circ(\alpha) <\lambda_{1}^{(1)}(\alpha)\quad\text{for}\quad \alpha\in(0,4\pi].
\]

Observe that \eqref{eq:SLk} is nothing more than the restriction of the eigenvalue problem \eqref{eq:alpha_BVP} (with $\Omega=\mathbf{D}$) to the $k$th Fourier mode.
Each angular frequency $k\in\mathbf{N}$ produces a one-dimensional Sturm--Liouville problem, the eigenvalues of which we denote by
\[
\{\lambda_j^{(k)}(\alpha)\}_{j\in\mathbf{N}}.
\]
We observe that since the potential $V_k(r)=k^2/r^2$ penalizes angular oscillations, the map $k\mapsto\lambda_j^{(k)}(\alpha)$ is strictly increasing for every $j\in\mathbf{N}$.
Consequently, the eigenvalues of \eqref{eq:alpha_BVP} on~$\mathbf{D}$ admit an ordered enumeration
\[
0<\lambda_1^{(0)}(\alpha)
<\lambda_1^{(1)}(\alpha)
<\lambda_1^{(2)}(\alpha)
<\cdots,
\]
where $\lambda_1^{(0)}(\alpha)<\Lambda_1^\circ(\alpha)$ corresponds to the radial optimizer,
and $\lambda_1^{(1)}(\alpha)=\lambda_2^\circ(\alpha)$ is the first nonradial eigenvalue.
The associated eigenspaces have multiplicity two for $k\ge 1$ (cosine and sine modes).

In summary, we have the following lemma:
\begin{lemma}[Mode gap on the disk]\label{lem:mode_gap_disk}
Let $\mathbf{D}\subset\mathbf{R}^2$ be the open unit disk and $\bar{u}_*\in\mathcal{M}_\alpha(\mathbf{D})$ be a radially symmetric positive optimizer.
The first nonradial eigenvalue
$\lambda_2^\circ(\alpha):=\min_{k\ge1}\lambda_1^{(k)}(\alpha)$ satisfies
\[
\Lambda_1^\circ(\alpha) < \lambda_{2}^\circ(\alpha) \quad {\rm for} \quad \alpha\in(0,4\pi].
\]
\end{lemma}

\begin{proof}
Since $\bar{u}_*\in\mathcal{M}_\alpha(\mathbf{D})$ is radial, the Euler--Lagrange equation reads
\[
\begin{cases}
-\bar{u}_*''(r)-\frac1r \bar{u}_*'(r)=\Lambda_1^\circ(\alpha)e^{\alpha \bar{u}_*^2(r)}\bar{u}_*(r) \quad {\rm in} \quad (0,1),\\
\bar{u}_*'(0)=0,\quad \bar{u}_*(1)=0.  
\end{cases}
\]
Setting $\psi(r):=\bar{u}_*'(r)$ and differentiating in $r$ gives us
\begin{equation}\label{eq:psi-k1}
-\psi''(r)-\frac1r\psi'(r)+\frac{1}{r^2}\psi(r)
=\Lambda_1^\circ(\alpha)(1+2\alpha \bar{u}_*^2(r))e^{\alpha \bar{u}_*^2(r)}\,\psi(r)
=\Lambda_1^\circ(\alpha)W(\bar{u}_*(r))\psi(r).
\end{equation}
Thus $\psi\in \mathcal{C}^2((0,1))$ solves the $k=1$ equation of \eqref{eq:SLk} in the open interval, with
$\psi(0)=0$ and $\psi(1)=\bar{u}_*'(1)<0$, where we used the Hopf lemma for $\bar{u}_*>0$.
Let $\phi_1$ be the positive first eigenfunction of \eqref{eq:SLk} with $k=1$,
normalized, so that $\phi_1>0$ in $(0,1)$ satisfies $\phi_1'(1)<0$ by Hopf’s boundary lemma and
\[
\begin{cases}
-(r\phi_1')'+\frac{1}{r}\phi_1=\lambda_1^{(1)}(\alpha)rW(\bar{u}_*(r))\phi_1 \quad {\rm in} \quad (0,1),\\
\phi_1(0)=\phi_1(1)=0. 
\end{cases}
\]

Next, rewriting \eqref{eq:psi-k1} and the equation for $\phi_1$ as
\[
\begin{cases}
(r\psi')'-\frac{1}{r}\psi=-\Lambda_1^\circ(\alpha)r W_\alpha(\bar{u}_*) \psi,\\
(r\phi_1')'-\frac{1}{r}\phi_1=-\lambda_1^{(1)}(\alpha)r W_\alpha(u_*)\phi_1.
\end{cases}
\]
From this, we get
\[
\int_0^1\big[\psi(r)(r\phi_1'(r))'-\phi_1(r)(r\psi'(r))'\big]\ud r
=\big(\Lambda_1^\circ(\alpha)-\lambda_1^{(1)}(\alpha)\big)\int_0^1 r W_\alpha(\bar{u}_*(r)) \psi(r)\phi_1(r)\ud r.
\]
We integrate the left-hand side by parts to find
\[
\big[r(\psi(r)\phi_1'(r)-\phi_1(r)\psi'(r))\big]_{0}^{1}
=\big(\Lambda_1^\circ(\alpha)-\lambda_1^{(1)}(\alpha)\big)\int_0^1 rW_\alpha(\bar{u}_*(r))\psi(r)\phi_1(r)\ud r.
\]
At $r=0$, we have $\psi(0)=\phi_1(0)=0$, and at $r=1$ we use $\phi_1(1)=0$ to obtain
\[
0<\psi(1)\phi_1'(1)=
\big(\Lambda_1^\circ(\alpha)-\lambda_1^{(1)}(\alpha)\big)\int_0^1 rW_\alpha(\bar{u}_*(r))\psi(r)\phi_1(r)\ud r.
\]
Since $\bar{u}_*\in \mathcal{C}^2((0,1))$ is strictly decreasing in $r$, it follows that $\psi=\bar{u}_*'<0$ and $\phi_1,W_\alpha(\bar{u}_*)>0$ on $(0,1)$, which implies $\psi(1)\phi_1'(1)>0$ and 
\[
\int_0^1 r W_\alpha(\bar{u}_*(r))\psi(r) \phi_1(r)\ud r<0 \quad {\rm for} \quad \alpha\in(0,4\pi].
\]
Therefore, $\Lambda_1^\circ(\alpha)-\lambda_1^{(1)}(\alpha)<0$, which yields
$\lambda_1^{(1)}(\alpha)>\Lambda_1^\circ(\alpha)$.
Finally, $\Lambda_1^\circ(\alpha)>\lambda_1^{(0)}(\alpha)$ follows from the 
fact that $e^{\alpha \bar{u}_*^2(r)}<W_\alpha(\bar{u}_*(r))$.
We conclude the proof of the lemma.
\end{proof}

\begin{remark}
This spectral hierarchy has a direct variational interpretation.
If we write the constrained second variation of the Trudinger--Moser functional at
$\bar{u}_*\in\mathcal{M}_\alpha(\mathbf{D})\cap\mathcal{S}_1$ as
\[
\mathscr{Q}_{\alpha,\mathbf{D}}[w]
=\int_{\mathbf{D}}|\nabla w|^2\ud x
-\Lambda_1^\circ(\alpha)\int_{\mathbf{D}}(1+2\alpha \bar{u}_*^2)e^{\alpha \bar{u}_*^2}w^2\ud x,
\]
then the decomposition $w=\sum_{k=0}^\infty w^{(k)}$ yields
\[
\mathscr{Q}_{\alpha,\mathbf{D}}[w]
=\sum_{k=0}^\infty
\big(\lambda_1^{(k)}(\alpha)-\Lambda_1^\circ(\alpha)\big)
\Lambda_1^\circ(\alpha)\int_{\mathbf{D}}(1+2\alpha u_*^2)e^{\alpha u_*^2}{w^{(k)}}^2\ud x.
\]
In particular, the positive spectral gap
$\lambda_1^{(1)}(\alpha)-\Lambda_1^\circ(\alpha)>0$
from Lemma~\ref{lem:mode_gap_disk}
provides a uniform lower bound on the Hessian 
\[
\mathscr{Q}_{\alpha,\mathbf{D}}[w] \ge
\big(\lambda_2^\circ(\alpha)-\Lambda_1^\circ(\alpha)\big)
\|\nabla w_\perp\|_{L^2(\mathbf{D})}^2,
\]
where $w_\perp=\sum_{k=1}^\infty w^{(k)}$ represents the projection onto all nonradial directions.
This inequality expresses the quadratic coercivity of the Trudinger--Moser functional
restricted to $\mathcal{S}_1$ and is the analytic core of the critical stability
result on the disk.
\end{remark}

\subsection{Symmetric critical case}\label{subsec:critical_case}
Now, we are ready to prove the main stability result, which is based on Lemma~\ref{lem:mode_gap_disk}.

\begin{proof}[Proof of Theorem~\ref{thm:critical_stability}]
Initially, we observe that by Lemma~\ref{lm:sphere_to_ball}, it is enough to work on the unit sphere $\mathcal{S}_1\subset W^{1,2}_0(\mathbf{D})$.
We divide the proof into two claims as follows:

\noindent{\bf Claim 1:} Any optimizer $\bar{u}_*\in \mathcal{M}_{\alpha}(\mathbf{D})\cap\mathcal{S}_1$ is positive and radially symmetric.

\noindent Since $\bar{u}_*\in \mathcal{M}_{\alpha}(\mathbf{D})\cap\mathcal{S}_1$ is a critical point of $\mathscr{E}_{\alpha,\mathbf{D}}$, it satisfies $\mathscr{E}_{\alpha,\mathbf{D}}(\bar{u}_*)={\rm TM}(\alpha,\mathbf{D})$, or equivalently
\begin{equation}\label{euler_lag_eqn_disk}
		 \begin{cases}
			-\Delta \bar{u}_*  = \Lambda_1^{\circ}(\alpha)  e^{\alpha \bar{u}_{*}^2} \bar{u}_* & \mbox{in} \quad \mathbf{D}\\
			\bar{u}_*=0 & \mbox{on} \quad\partial \mathbf{D}.
		\end{cases} 
	\end{equation}
Hence, since $f_\alpha(s)=\Lambda_1^{\circ}(\alpha)  e^{\alpha s^2} s$ satisfies $f_\alpha\in \mathcal{C}^1(\mathbf{R})$, by the classical result of Gidas, Ni, and Nirenberg~\cite[Theorem~1]{MR634248}, it follows that $\bar{u}_*\in \mathcal{C}^2(\mathbf{D})\cap W^{1,2}_{0,{\rm rad}}(\mathbf{D})$ is radially symmetric and nonincreasing, {\it i.e.,} $\partial_r \bar{u}_*\leq 0$ for $r>0$. 
Moreover, the positivity $\bar{u}_*>0$ is a direct consequence of the standard maximum principle.

\noindent{\bf Claim 2:} For all $\alpha\in(0,4\pi]$ and $u\in\mathcal{S}_1$, it holds
\[
{\rm def}_{\alpha,\mathbf{D}}(u) \ge C(\alpha)\,\|\nabla(u-\bar{u}_*)\|_{L^2(\mathbf{D})}^2.
\]

\noindent Initially, we decompose any $w\in T_{\bar u_*}\mathcal{S}_1$ into angular modes
$w=\sum_{k=0}^\infty w^{(k)}$, which are orthogonal in $L^2(\mathbf{D})$ and in the Dirichlet form.
The quadratic form \eqref{eq:Qalpha-disk} splits modewise as
\[
\mathscr{Q}_{\alpha,\mathbf{D}}[w]=\sum_{k=0}^\infty\mathscr{Q}_{\alpha,\mathbf{D}}[w^{(k)}].
\]

Subsequently, we divide the rest of the proof into some steps with respect to the Fourier eigenmodes. 

\noindent\textit{Step 1: Radial sector ($k=0$).}
By Lemma~\ref{lem:Qalpha-kernel-extension}\textup{(iii)}, the first eigenfunction $v_*>0$ of \eqref{eq:alpha_BVP} satisfies $v_*\notin T_{\bar u_*}\mathcal{S}_1$.
Since $v_*$ is the unique (up to scaling) eigenfunction with eigenvalue $\lambda_1^{(0)}(\alpha)<\Lambda_1^\circ(\alpha)$,
the tangent space constraint projects out the only direction on which $\mathscr{Q}_{\alpha,\mathbf{D}}$ is negative.
More precisely, by the Courant--Fischer min-max characterization (cf.\ \cite[Chapter~VI]{CourantHilbert}),
the Rayleigh quotient
\[
\mathcal{R}_{\alpha,\mathbf{D}}(w):=\frac{\int_{\mathbf{D}}|\nabla w^{(0)}|^2\ud x}{\int_{\mathbf{D}}(1+2\alpha\bar u_*^2)e^{\alpha\bar u_*^2}(w^{(0)})^2\ud x},
\]
restricted to $T_{\bar u_*}\mathcal{S}_1\cap W^{1,2}_{0,{\rm rad}}(\mathbf{D})$, which has codimension one in the radial sector, with the removed direction spanned by $v_*$ satisfies 
$\mathcal{R}_{\alpha,\mathbf{D}}(w)\geq\lambda_2^{(0)}(\alpha)$.  Next, using that 
\[
\mathscr{Q}_{\alpha,\mathbf{D}}[w^{(0)}]=\|\nabla w^{(0)}\|_{L^2(\mathbf{D})}^2-\Lambda_1^\circ(\alpha)\int_{\mathbf{D}}(1+2\alpha\bar u_*^2)e^{\alpha\bar u_*^2}(w^{(0)})^2\ud x,
\]
one has
\begin{equation}\label{eq:radial-coercivity}
\mathscr{Q}_{\alpha,\mathbf{D}}[w^{(0)}]\ge \left(1-\frac{\Lambda_1^\circ(\alpha)}{\lambda_2^{(0)}(\alpha)}\right)\|\nabla w^{(0)}\|_{L^2(\mathbf{D})}^2
\quad\text{for all} \quad w^{(0)}\in T_{\bar u_*}\mathcal{S}_1\cap W^{1,2}_{0,{\rm rad}}(\mathbf{D}),
\end{equation}
where $\lambda_2^{(0)}(\alpha)$ is the second radial Sturm--Liouville eigenvalue.
Finally, since 
\[
\lambda_2^{(0)}(\alpha)\to\mu_2^{(0)}=j_{0,2}^2\approx 30.47 \quad {\rm as} \quad \alpha\to0, \quad {\rm and} \quad \Lambda_1^\circ(\alpha)\le\mu_1=j_{0,1}^2\approx 5.78,
\]
the strict inequality $\lambda_2^{(0)}(\alpha)>\Lambda_1^\circ(\alpha)$ holds for all $\alpha\in(0,4\pi]$.

\noindent\textit{Step 2: Nonradial sectors ($k\ge 1$).}
On each nonradial sector $k\ge1$, Lemma~\ref{lem:mode_gap_disk} gives us a uniform spectral gap
\[
\mathscr{Q}_{\alpha,\mathbf{D}}[w^{(k)}]
 \ge \big(\lambda_1^{(k)}(\alpha)-\Lambda_1^\circ(\alpha)\big)\| \nabla w^{(k)}\|_{L^2(\mathbf{D})}^2
\ge c_\circ(\alpha)\|\nabla w^{(k)}\|_{L^2(\mathbf{D})}^2,
\]
with $c_\circ(\alpha):=\lambda_2^\circ(\alpha)-\Lambda_1^\circ(\alpha)>0$.

\noindent\textit{Step 3: Full coercivity.}
Combining \eqref{eq:radial-coercivity} with the nonradial estimate and setting
\[
\tilde c(\alpha):=\min\big\{\lambda_2^{(0)}(\alpha)-\Lambda_1^\circ(\alpha),\,c_\circ(\alpha)\big\}>0,
\]
a direct summation yields
\begin{equation}\label{eq:coercivity-disk}
\mathscr{Q}_{\alpha,\mathbf{D}}[w]\ge \tilde c(\alpha)\|\nabla w\|_{L^2(\mathbf{D})}^2
\quad\text{for all }w\in T_{\bar u_*}\mathcal{S}_1.
\end{equation}

For $u\in\mathcal{S}_1$, we write $u=\bar u_*+w$, with $w\in T_{\bar u_*}\mathcal{S}_1$.
By the Taylor expansion around the optimizer in Lemma~\ref{lm:taylor_expansion}, it follows
\[
\mathscr{E}_{\alpha,\mathbf{D}}(\bar u_*)-\mathscr{E}_{\alpha,\mathbf{D}}(u)
=\frac{\alpha}{\Lambda_1^\circ(\alpha)}\mathscr{Q}_\alpha[w]
+\mathcal{O}\big(\|\nabla w\|_{L^2(\mathbf{D})}^3\big) \quad {\rm as} \quad \|\nabla w\|_{L^2(\mathbf{D})}\to 0.
\]
Using \eqref{eq:coercivity-disk}, for $0<\|\nabla w\|_{L^2(\mathbf{D})}\ll 1$ sufficiently small one obtains $C_\circ(\alpha)>0$ such that
\[
{\rm def}_{\alpha,\mathbf{D}}(u)
 \ge C_\circ(\alpha)\|\nabla(u-\bar u_*)\|_{L^2(\mathbf{D})}^2 \quad {\rm for \ all} \quad u\in\mathcal{S}_1.
\]
The proof of the theorem is concluded.
\end{proof}

\section*{Declarations}

\bigskip

\begin{flushleft}
	{\bf Ethical Approval:}  Not applicable.\\
	{\bf Competing interests:}  Not applicable. \\
	{\bf Authors' contributions:}  All authors contributed equally to the results and writing of the manuscript.\\
	{\bf Availability of data and material:}  All data generated or analyzed during this study are included in this article.\\
	{\bf Consent to participate:}  All authors consent to participate in this work.\\
	{\bf Conflict of interest:} The authors declare that they have no conflict of interest. \\
	{\bf Consent for publication:}  All authors consent for publication. \\
\end{flushleft}

\bigskip

\bibliography{references}

@article {Lions85,
    AUTHOR = {Lions, P.-L.},
     TITLE = {The concentration-compactness principle in the calculus of
              variations. {T}he limit case. {I}},
   JOURNAL = {Rev. Mat. Iberoamericana},
  FJOURNAL = {Revista Matem\'{a}tica Iberoamericana},
    VOLUME = {1},
      YEAR = {1985},
    NUMBER = {1},
     PAGES = {145--201},
      ISSN = {0213-2230},
   MRCLASS = {49A22 (58E30)},
  MRNUMBER = {834360},
MRREVIEWER = {Ll. G. Chambers},
       DOI = {10.4171/RMI/6},
       URL = {https://doi.org/10.4171/RMI/6},
}

@article {Trudinger67,
    AUTHOR = {Trudinger, Neil S.},
     TITLE = {On imbeddings into {O}rlicz spaces and some applications},
   JOURNAL = {J. Math. Mech.},
    VOLUME = {17},
      YEAR = {1967},
     PAGES = {473--483},
   MRCLASS = {46.38},
  MRNUMBER = {0216286},
MRREVIEWER = {J. Albrycht},
       DOI = {10.1512/iumj.1968.17.17028},
       URL = {https://doi.org/10.1512/iumj.1968.17.17028},
}

@article {Carleson-Chang,
    AUTHOR = {Carleson, Lennart and Chang, Sun-Yung A.},
     TITLE = {On the existence of an extremal function for an inequality of
              {J}. {M}oser},
   JOURNAL = {Bull. Sci. Math. (2)},
  FJOURNAL = {Bulletin des Sciences Math\'{e}matiques. 2e S\'{e}rie},
    VOLUME = {110},
      YEAR = {1986},
    NUMBER = {2},
     PAGES = {113--127},
      ISSN = {0007-4497},
   MRCLASS = {46E35 (46E30)},
  MRNUMBER = {878016},
MRREVIEWER = {H. Triebel},
}

@article {Moser1970/71,
	AUTHOR = {Moser, J.},
	TITLE = {A sharp form of an inequality by {N}. {T}rudinger},
	JOURNAL = {Indiana Univ. Math. J.},
	FJOURNAL = {Indiana University Mathematics Journal},
	VOLUME = {20},
	YEAR = {1970/71},
	PAGES = {1077--1092},
	ISSN = {0022-2518},
	MRCLASS = {46E35},
	MRNUMBER = {0301504},
	MRREVIEWER = {G. O. Okikiolu},
	DOI = {10.1512/iumj.1971.20.20101},
	URL = {https://doi.org/10.1512/iumj.1971.20.20101},
}

@article {Flucher1992,
    AUTHOR = {Flucher, Martin},
     TITLE = {Extremal functions for the {T}rudinger-{M}oser inequality in
              {$2$} dimensions},
   JOURNAL = {Comment. Math. Helv.},
  FJOURNAL = {Commentarii Mathematici Helvetici},
    VOLUME = {67},
      YEAR = {1992},
    NUMBER = {3},
     PAGES = {471--497},
      ISSN = {0010-2571},
   MRCLASS = {58E35 (49Q15)},
  MRNUMBER = {1171306},
MRREVIEWER = {Helmut Kaul},
       DOI = {10.1007/BF02566514},
       URL = {https://doi.org/10.1007/BF02566514},
}

@article {Lin1996,
    AUTHOR = {Lin, Kai-Ching},
     TITLE = {Extremal functions for {M}oser's inequality},
   JOURNAL = {Trans. Amer. Math. Soc.},
  FJOURNAL = {Transactions of the American Mathematical Society},
    VOLUME = {348},
      YEAR = {1996},
    NUMBER = {7},
     PAGES = {2663--2671},
      ISSN = {0002-9947},
   MRCLASS = {58E35 (46E35 49Q20)},
  MRNUMBER = {1333394},
MRREVIEWER = {Tero Kilpel\"{a}inen},
       DOI = {10.1090/S0002-9947-96-01541-3},
       URL = {https://doi.org/10.1090/S0002-9947-96-01541-3},
}

@article{CarlenFigalli2013,
 author = {Carlen, Eric A. and Figalli, Alessio},
 title = {Stability for a {GNS} inequality and the log-{HLS} inequality, with application to the critical mass {Keller}-{Segel} equation},
 fjournal = {Duke Mathematical Journal},
 journal = {Duke Math. J.},
 issn = {0012-7094},
 volume = {162},
 number = {3},
 pages = {579--625},
 year = {2013},
 language = {English},
 doi = {10.1215/00127094-2019931},
 zbMATH = {6145496},
 Zbl = {1307.26027}
}

@incollection{FrankKonigTang2023,
 author = {Frank, Rupert L.},
 title = {The sharp {Sobolev} inequality and its stability: an introduction},
 booktitle = {Geometric and analytic aspects of functional variational principles. Cetraro, Italy, June 2022. Lecture notes},
 isbn = {978-3-031-67600-0; 978-3-031-67601-7},
 pages = {1--64},
 year = {2024},
 publisher = {Cham: Springer; Florence: Fondazione CIME},
 language = {English},
 doi = {10.1007/978-3-031-67601-7_1},
 zbMATH = {8018427},
 Zbl = {1564.49012}
}

@article{FigalliZhang2022,
 author = {Figalli, Alessio and Zhang, Yi Ru-Ya},
 title = {Sharp gradient stability for the {Sobolev} inequality},
 fjournal = {Duke Mathematical Journal},
 journal = {Duke Math. J.},
 issn = {0012-7094},
 volume = {171},
 number = {12},
 pages = {2407--2459},
 year = {2022},
 language = {English},
 doi = {10.1215/00127094-2022-0051},
 zbMATH = {7600543},
 Zbl = {1504.46040}
}

@article{DEFFL2025,
 author = {Dolbeault, Jean and Esteban, Maria J. and Figalli, Alessio and Frank, Rupert L. and Loss, Michael},
 title = {Sharp stability for {Sobolev} and log-{Sobolev} inequalities, with optimal dimensional dependence},
 fjournal = {Cambridge Journal of Mathematics},
 journal = {Camb. J. Math.},
 issn = {2168-0930},
 volume = {13},
 number = {2},
 pages = {359--430},
 year = {2025},
 language = {English},
 doi = {10.4310/CJM.250325022725},
 zbMATH = {8018035},
 Zbl = {1561.49020}
}

@article{BE,
title={A note on the {S}obolev inequality},
author={G. Bianchi and H. Egnell}, 
journal={J. Funct. Anal.},
volume={100},
year={1991}, 
pages={18--24}
}

@book {MR2759829,
    AUTHOR = {Brezis, Haim},
     TITLE = {Functional analysis, {S}obolev spaces and partial differential
              equations},
    SERIES = {Universitext},
 PUBLISHER = {Springer, New York},
      YEAR = {2011},
     PAGES = {xiv+599},
      ISBN = {978-0-387-70913-0},
   MRCLASS = {35-01 (46-01 46E35 46N20 47F05)},
  MRNUMBER = {2759829},
MRREVIEWER = {Vicen\c{t}iu D. R\u{a}dulescu},
}

@article {MR790771,
    AUTHOR = {Brezis, Ha\"{\i}m and Lieb, Elliott H.},
     TITLE = {Sobolev inequalities with remainder terms},
   JOURNAL = {J. Funct. Anal.},
  FJOURNAL = {Journal of Functional Analysis},
    VOLUME = {62},
      YEAR = {1985},
    NUMBER = {1},
     PAGES = {73--86},
      ISSN = {0022-1236},
   MRCLASS = {46E35},
  MRNUMBER = {790771},
MRREVIEWER = {H. Triebel},
       DOI = {10.1016/0022-1236(85)90020-5},
       URL = {https://doi.org/10.1016/0022-1236(85)90020-5},
}

@incollection {MR634248,
    AUTHOR = {Gidas, B. and Ni, Wei Ming and Nirenberg, L.},
     TITLE = {Symmetry of positive solutions of nonlinear elliptic equations
              in {${\bf R}\sp{n}$}},
 BOOKTITLE = {Mathematical analysis and applications, {P}art {A}},
    SERIES = {Adv. Math. Suppl. Stud.},
    VOLUME = {7},
     PAGES = {369--402},
 PUBLISHER = {Academic Press, New York-London},
      YEAR = {1981},
   MRCLASS = {35J60 (53C05 58G20)},
  MRNUMBER = {634248},
MRREVIEWER = {D. E. Edmunds},
}

@article {MR4427104,
    AUTHOR = {Engelstein, Max and Neumayer, Robin and Spolaor, Luca},
     TITLE = {Quantitative stability for minimizing {Y}amabe metrics},
   JOURNAL = {Trans. Amer. Math. Soc. Ser. B},
  FJOURNAL = {Transactions of the American Mathematical Society. Series B},
    VOLUME = {9},
      YEAR = {2022},
     PAGES = {395--414},
   MRCLASS = {53C18 (26D10 58E05)},
  MRNUMBER = {4427104},
MRREVIEWER = {Ilaria Mondello},
       DOI = {10.1090/btran/111},
       URL = {https://doi.org/10.1090/btran/111},
}

@article {MR185604,
    AUTHOR = {Smale, S.},
     TITLE = {An infinite dimensional version of {S}ard's theorem},
   JOURNAL = {Amer. J. Math.},
  FJOURNAL = {American Journal of Mathematics},
    VOLUME = {87},
      YEAR = {1965},
     PAGES = {861--866},
      ISSN = {0002-9327},
   MRCLASS = {57.55 (57.50)},
  MRNUMBER = {185604},
MRREVIEWER = {Richard Beals},
       DOI = {10.2307/2373250},
       URL = {https://doi.org/10.2307/2373250},
}

@book{Zeidler1985,
  author    = {Eberhard Zeidler},
  title     = {Applied Functional Analysis: Applications to Mathematical Physics},
  series    = {Applied Mathematical Sciences},
  volume    = {108},
  publisher = {Springer},
  address   = {New York},
  year      = {1985},
  note      = {}
}

@article {MR4623703,
    AUTHOR = {K\"{o}nig, Tobias},
     TITLE = {On the sharp constant in the {B}ianchi-{E}gnell stability
              inequality},
   JOURNAL = {Bull. Lond. Math. Soc.},
  FJOURNAL = {Bulletin of the London Mathematical Society},
    VOLUME = {55},
      YEAR = {2023},
    NUMBER = {4},
     PAGES = {2070--2075},
      ISSN = {0024-6093},
   MRCLASS = {26D10 (58E35)},
  MRNUMBER = {4623703},
MRREVIEWER = {Alireza Ranjbar-Motlagh},
       DOI = {10.1112/blms.12837},
       URL = {https://doi.org/10.1112/blms.12837},
}

@misc{arxiv:2211.14185,
      title={Stability for the {S}obolev inequality: existence of a minimizer}, 
      author={Tobias König},
      journal={J. Eur. Math. Soc.},
      year={2025},
      eprint={2211.14185},
      archivePrefix={arXiv},
      primaryClass={math.AP},
      url={https://arxiv.org/abs/2211.14185}, 
}

@misc{andrade2024quantitativestabilitytotalqcurvature,
      title={Quantitative stability of the total $Q$-curvature near minimizing metrics}, 
      author={João Henrique Andrade and Tobias König and Jesse Ratzkin and Juncheng Wei},
      note={arXiv:2407.06934 [math.AP]},
      year={2024},
      eprint={2407.06934},
      archivePrefix={arXiv},
      primaryClass={math.AP},
      url={https://arxiv.org/abs/2407.06934}, 
}

@article{CFW,
 Author = {Chen, Shibing and Frank, Rupert L. and Weth, Tobias},
 Title = {Remainder terms in the fractional {Sobolev} inequality},
 FJournal = {Indiana University Mathematics Journal},
 Journal = {Indiana Univ. Math. J.},
 ISSN = {0022-2518},
 Volume = {62},
 Number = {4},
 Pages = {1381--1397},
 Year = {2013},
 Language = {English},
 DOI = {10.1512/iumj.2013.62.5065},
 URL = {resolver.caltech.edu/CaltechAUTHORS:20170502-150250004},
 zbMATH = {6293939},
 Zbl = {1296.46032}
}

@article{LW,
 Author = {Lu, Guozhen and Wei, Juncheng},
 Title = {On a {Sobolev} inequality with remainder terms},
 FJournal = {Proceedings of the American Mathematical Society},
 Journal = {Proc. Am. Math. Soc.},
 ISSN = {0002-9939},
 Volume = {128},
 Number = {1},
 Pages = {75--84},
 Year = {2000},
 Language = {English},
 DOI = {10.1090/S0002-9939-99-05497-0},
 zbMATH = {1355927},
 Zbl = {0961.35100}
}

@misc{deFigueiredo1982,
 author = {de Figueiredo, Djairo Guedes},
 title = {Positive solutions of semilinear elliptic problems},
 year = {1982},
 language = {English},
 howpublished = {Differential equations, {Proc}., {Sao} {Paulo} 1981, {Lect}. {Notes} {Math}. 957, 34-87 (1982).},
 doi = {10.1007/bfb0066233},
 zbMATH = {3798268},
 Zbl = {0506.35038}
}

@article{Cerny2013,
 author = {{\v{C}}ern{\'y}, Robert and Cianchi, Andrea and Hencl, Stanislav},
 title = {Concentration-compactness principles for {Moser}-{Trudinger} inequalities: new results and proofs},
 fjournal = {Annali di Matematica Pura ed Applicata. Serie Quarta},
 journal = {Ann. Mat. Pura Appl. (4)},
 issn = {0373-3114},
 volume = {192},
 number = {2},
 pages = {225--243},
 year = {2013},
 language = {English},
 doi = {10.1007/s10231-011-0220-3},
 zbMATH = {6153611},
 Zbl = {1272.46023}
}

@article{LuLu2020,
 author = {Chen, Lu and Lu, Guozhen and Zhu, Maochun},
 title = {Existence and nonexistence of extremals for critical {Adams} inequalities in {{\(\mathbb{R}^4\)}} and {Trudinger}-{Moser} inequalities in {{\(\mathbb{R}^2\)}}},
 fjournal = {Advances in Mathematics},
 journal = {Adv. Math.},
 issn = {0001-8708},
 volume = {368},
 pages = {60},
 note = {Id/No 107143},
 year = {2020},
 language = {English},
 doi = {10.1016/j.aim.2020.107143},
 zbMATH = {7201202},
 Zbl = {1441.35121}
}

@article {MR2560131,
    AUTHOR = {Micheletti, Anna Maria and Pistoia, Angela},
     TITLE = {Generic properties of singularly perturbed nonlinear elliptic
              problems on {R}iemannian manifold},
   JOURNAL = {Adv. Nonlinear Stud.},
  FJOURNAL = {Advanced Nonlinear Studies},
    VOLUME = {9},
      YEAR = {2009},
    NUMBER = {4},
     PAGES = {803--813},
      ISSN = {1536-1365},
   MRCLASS = {53C21 (58E05)},
  MRNUMBER = {2560131},
MRREVIEWER = {Youssef Maliki},
       DOI = {10.1515/ans-2009-0411},
       URL = {https://doi.org/10.1515/ans-2009-0411},
}

@article {MR2982783,
    AUTHOR = {Ghimenti, Marco and Micheletti, Anna Maria},
     TITLE = {Non degeneracy for solutions of singularly perturbed nonlinear
              elliptic problems on symmetric {R}iemannian manifolds},
   JOURNAL = {Commun. Pure Appl. Anal.},
  FJOURNAL = {Communications on Pure and Applied Analysis},
    VOLUME = {12},
      YEAR = {2013},
    NUMBER = {2},
     PAGES = {679--693},
      ISSN = {1534-0392},
   MRCLASS = {58J05 (35B25 35J61 35R01)},
  MRNUMBER = {2982783},
MRREVIEWER = {Mircea Crasmareanu},
       DOI = {10.3934/cpaa.2013.12.679},
       URL = {https://doi.org/10.3934/cpaa.2013.12.679},
}

@article {deFigueiredo2002,
    AUTHOR = {de Figueiredo, D. G. and do \'O, J. M. and Ruf, B.},
     TITLE = {On an inequality by {N}. {T}rudinger and {J}. {M}oser and
              related elliptic equations},
   JOURNAL = {Comm. Pure Appl. Math.},
  FJOURNAL = {Communications on Pure and Applied Mathematics},
    VOLUME = {55},
      YEAR = {2002},
    NUMBER = {2},
     PAGES = {135--152},
      ISSN = {0010-3640},
   MRCLASS = {35J60 (35B33 46E35)},
  MRNUMBER = {1857881},
MRREVIEWER = {Maria Assunta Pozio},
       DOI = {10.1002/cpa.10015},
       URL = {https://doi.org/10.1002/cpa.10015},
}

@book{CourantHilbert,
    AUTHOR = {Courant, Richard and Hilbert, David},
     TITLE = {Methods of Mathematical Physics. {V}ol.\ {I}},
 PUBLISHER = {Interscience Publishers, Inc., New York, N.Y.},
      YEAR = {1953},
     PAGES = {xv+561},
  MRNUMBER = {0065391},
}

@article {Wei_Wu,
    AUTHOR = {Wei, Juncheng and Wu, Yuanze},
     TITLE = {On the stability of the {C}affarelli-{K}ohn-{N}irenberg
              inequality},
   JOURNAL = {Math. Ann.},
  FJOURNAL = {Mathematische Annalen},
    VOLUME = {384},
      YEAR = {2022},
    NUMBER = {3-4},
     PAGES = {1509--1546},
      ISSN = {0025-5831,1432-1807},
   MRCLASS = {35B09 (35B33 35B40 35J20)},
  MRNUMBER = {4498479},
       DOI = {10.1007/s00208-021-02325-0},
       URL = {https://doi.org/10.1007/s00208-021-02325-0},
}

@article {Frank_Peteranderl,
    AUTHOR = {Frank, Rupert L. and Peteranderl, Jonas W.},
     TITLE = {Degenerate stability of the {C}affarelli-{K}ohn-{N}irenberg
              inequality along the {F}elli-{S}chneider curve},
   JOURNAL = {Calc. Var. Partial Differential Equations},
  FJOURNAL = {Calculus of Variations and Partial Differential Equations},
    VOLUME = {63},
      YEAR = {2024},
    NUMBER = {2},
     PAGES = {Paper No. 44, 33},
      ISSN = {0944-2669,1432-0835},
   MRCLASS = {35J20 (35A23 46E35)},
  MRNUMBER = {4695793},
MRREVIEWER = {Alireza\ Ranjbar-Motlagh},
       DOI = {10.1007/s00526-023-02641-0},
       URL = {https://doi.org/10.1007/s00526-023-02641-0},
}

@article {CLT,
    AUTHOR = {Chen, Lu and Lu, Guozhen and Tang, Hanli},
     TITLE = {Stability of {H}ardy-{L}ittlewood-{S}obolev inequalities with
              explicit lower bounds},
   JOURNAL = {Adv. Math.},
  FJOURNAL = {Advances in Mathematics},
    VOLUME = {450},
      YEAR = {2024},
     PAGES = {Paper No. 109778, 28},
      ISSN = {0001-8708,1090-2082},
   MRCLASS = {46E35 (26D10 35A23)},
  MRNUMBER = {4758206},
MRREVIEWER = {Kanailal\ Mahato},
       DOI = {10.1016/j.aim.2024.109778},
       URL = {https://doi.org/10.1016/j.aim.2024.109778},
}
\bibliographystyle{abbrv}

\end{document}